\documentclass[10pt]{amsart}

\usepackage[english]{babel}
\usepackage{amsmath}
\usepackage{amssymb, amscd, amsthm}
\usepackage{mathrsfs}
\usepackage{mathscinet}
\usepackage{cases}
\usepackage{multicol}
\usepackage{enumerate}
\usepackage{stackrel}
\usepackage[all]{xy}
\usepackage[margin=1.2in]{geometry}
\usepackage{stmaryrd}
\usepackage{graphicx}
\usepackage{url}
\usepackage{hyperref}
\hypersetup{
  linktoc=page,
}

\newtheorem{satz}{Satz}[section]
\newtheorem{Theorem}[satz]{Theorem}
\newtheorem{Cor}[satz]{Corollary}
\newtheorem{Lemma}[satz]{Lemma}
\newtheorem{Prop}[satz]{Proposition}

\theoremstyle{definition}
\newtheorem{Def}[satz]{Definition}

\newtheorem{Remark}[satz]{Remark}
\newtheorem{Example}[satz]{Example}
\newtheorem*{Def*}{Definition}
\newtheorem*{Example*}{Example}
\newtheorem*{Theorem*}{Theorem}

\newcommand{\R}{\mathbb{R}}

\newcommand{\upi}{\underline{\pi}}
\newcommand{\N}{\mathbb{N}}
\newcommand{\F}{\mathcal{F}}
\newcommand{\U}{\mathcal{U}}
\newcommand{\xr}{\xrightarrow}
\newcommand{\T}{\mathbf{T}}
\newcommand{\map}{\operatorname{map}}

\newcommand{\Hom}{\operatorname{Hom}}
\newcommand{\I}{\mathbf{I}}

\makeatletter
\newcommand*{\defeq}{\mathrel{\rlap{%
                     \raisebox{0.3ex}{$\m@th\cdot$}}%
                     \raisebox{-0.3ex}{$\m@th\cdot$}}%
                     =}
\makeatother

\DeclareMathOperator*{\colim}{colim}

\DeclareMathOperator*{\tr}{tr}
\DeclareMathOperator*{\Inj}{Inj}
\DeclareMathOperator*{\Bij}{Bij}

\DeclareMathOperator*{\Lin}{L}
\DeclareMathOperator*{\triv}{triv}

\DeclareMathOperator*{\Cyl}{Cyl}

\DeclareMathOperator*{\Fin}{Fin}

\DeclareMathOperator*{\HEM}{H}
\DeclareMathOperator*{\sh}{sh}
\DeclareMathOperator*{\id}{id}

\numberwithin{equation}{section}

\title[Symmetric spectra model global homotopy theory]{Symmetric spectra model global homotopy theory of finite groups}
\author{Markus Hausmann}
\address{University of Bonn, Germany}
\email{hausmann@math.uni-bonn.de}

\begin{document}
\begin{abstract} In this paper we show that the category of symmetric spectra can be used to model global equivariant homotopy theory of finite groups.
\end{abstract}
\maketitle
\setcounter{section}{-1}
\section{Introduction}
Equivariant stable homotopy theory deals with the study of equivariant spectra and the cohomology theories they represent. While some of these equivariant theories are specific to a fixed group, many of them are defined in a uniform way for all compact Lie groups simultaneously, for example equivariant $K$-theory, Borel cohomology, equivariant bordism or equivariant cohomotopy. The idea of global equivariant homotopy theory is to view such a compatible collection of equivariant spectra - ranging through all compact Lie groups - as one ``global'' object, in particular to capture its full algebraic structure of restrictions, transfer maps and power operations. There have been various approaches to formalizing this idea and to obtain a category of global equivariant spectra, for example in \cite[Chapter 2]{LMS86}, \cite[Section 5]{GM97} and \cite{Boh14}.
In \cite{Sch15}, Schwede introduced a new approach by looking at the well-known category of orthogonal spectra of \cite{MMSS01} from a different point of view: Every orthogonal spectrum gives rise to a $G$-orthogonal spectrum for any compact Lie group $G$ by endowing it with the trivial $G$-action (and, classically, changing from the trivial to a complete $G$-universe, but this change of universe is an equivalence of categories on the point-set level). The fundamental observation used in \cite{Sch15} is that the $G$-homotopy type of such a $G$-orthogonal spectrum with trivial action is not determined by its non-equivariant homotopy type. There are maps of orthogonal spectra that are a non-equivariant stable equivalence but not a $G$-stable equivalence when given the trivial $G$-action.
Taking these $G$-homotopy types for varying $G$ into account gives rise to a much finer notion of weak equivalence called \emph{global equivalence} and thereby to the global stable homotopy category, which splits each non-equivariant homotopy type into many global variants.
A strength of Schwede's approach is that it on the one hand allows many examples (all the theories mentioned above 
are represented by a single orthogonal spectrum in this sense) and on the other hand is technically easy to work with, since the underlying category is just that of orthogonal spectra.

The purpose of this paper is to show that the category of symmetric spectra introduced by Hovey, Shipley and Smith in \cite{HSS00} can also be used to model global equivariant homotopy theory, if one takes ``global'' to mean all \emph{finite} groups instead of all compact Lie groups. Symmetric spectra have the advantage that they can also be based on simplicial sets and are generally more combinatorial, as it is sometimes easier to construct actions of symmetric groups than of orthogonal groups. For example, in \cite{Sch13alg} Schwede introduces a construction of the global algebraic $K$-theory of a ring (or more generally, of a category with an action of a specific $E_{\infty}$-operad), whose output is a symmetric spectrum and usually not an orthogonal spectrum.

Besides the fully global theory of orthogonal spectra that takes into account all compact Lie groups, \cite{Sch15} also contains a variant where only a fixed family of groups is considered. In particular, there is a version for the family of finite groups $\F in$. Then the main result of this paper can be stated as:
\begin{Theorem*}[Theorems \ref{theo:stablemod} and \ref{theo:quillen}] There exists a model structure on the category of symmetric spectra of spaces or simplicial sets - called the \emph{global} model structure - which is Quillen equivalent to orthogonal spectra with the $\F in$-global model structure of \cite{Sch15}.
\end{Theorem*}
More precisely, the forgetful functor from orthogonal to symmetric spectra is the right adjoint of a Quillen equivalence. The central notion in the global model structure is that of a \emph{global equivalence of symmetric spectra}. The basic idea is the same as for orthogonal spectra: Every symmetric spectrum gives rise to a $G$-symmetric spectrum for any finite group $G$ via the trivial action, in particular it can be evaluated on all finite $G$-sets and one can define its equivariant homotopy groups. However, unlike for orthogonal spectra, equivariant homotopy groups cannot be used to describe global equivalences -- a phenomenon already present for non-equivariant symmetric spectra and for $G$-symmetric spectra over a fixed finite group $G$. Instead we make use of the notion of $G$-stable equivalence introduced in \cite{Hau14} and define a map of symmetric spectra to be a global equivalence if for all finite groups $G$ it becomes a $G$-stable equivalence when given the trivial $G$-action.
The more complicated definition of $G$-stable equivalence and hence global equivalence is the main technical difference to orthogonal spectra. The work in this paper lies in assembling the model structures of \cite{Hau14} for varying $G$ into a global one, for which Proposition \ref{prop:fixedgomega} is central.

The cofibrations in our model structure are the same as in Shipley's flat (or $\mathbb{S}$-) model structure introduced in \cite{Shi04}, which hence forms a left Bousfield localization of ours. This determines the model structure completely; the fibrant objects can be characterized as global equivariant versions of $\Omega$-spectra (Definition \ref{def:globalomega}), similarly as for orthogonal spectra.  We further show that the global model structure (or a positive version) lifts to the categories of symmetric ring spectra and commutative symmetric ring spectra (called ``ultracommutative'' in \cite{Sch15}), and more generally to categories of modules, algebras and commutative algebras over a fixed (commutative) symmetric ring spectrum.

While equivariant homotopy groups of symmetric spectra cannot be used to characterize global equivalences, they nevertheless provide an important tool. We describe some of their properties and their functoriality as the group varies. This functoriality turns out to be more involved than for orthogonal spectra, as it interacts non-trivially with the theory of (global equivariant) semistability, i.e., the relationship between ``naive'' and derived equivariant homotopy groups of symmetric spectra.

Throughout we focus on the class of all finite groups, but symmetric spectra can also be used to model global homotopy theory with respect to smaller families of groups, such as abelian finite groups or $p$-groups for a fixed prime $p$. In Appendix \ref{sec:appendix} we give a short treatment of the modifications needed to obtain such a relative theory.

The paper is \textbf{organized} as follows: In Section \ref{sec:def} we recall the definition of symmetric spectra, explain how to evaluate them on finite $G$-sets (Section \ref{sec:evaluation}) and introduce global free spectra (Section \ref{sec:free}). Section \ref{sec:model} starts with the construction of the global level model structure (Proposition \ref{prop:levmod}), introduces global equivalences (Definition \ref{def:globalequivalence}) and global $\Omega$-spectra (Definition \ref{def:globalomega}), explains the connection between the two (Proposition \ref{prop:fixedgomega}) and finally contains a proof of the stable global model structure (Theorem \ref{theo:stablemod}). In Section \ref{sec:monoidal} we construct global model structures on module, algebra and commutative algebra categories. Section \ref{sec:homgroups} deals with equivariant homotopy groups of symmetric spectra.
Their definition is given in Section \ref{sec:defhomgroups}, their functoriality is explained in Sections \ref{sec:functoriality}, \ref{sec:restrictions} and \ref{sec:transfers} and the properties of globally semistable symmetric spectra are discussed in Section \ref{sec:semistability}. In Section \ref{sec:comparison} we prove that our model structure is Quillen equivalent to $\F in$-global orthogonal spectra. Section \ref{sec:exa} discusses examples of symmetric spectra from the global point of view. Finally, Appendix \ref{sec:appendix} deals with global homotopy theory of symmetric spectra with respect to a family of finite groups. 

\textbf{Acknowledgements}: I thank my advisor Stefan Schwede for suggesting this project and for many helpful discussions and comments. This research was supported by the Deutsche Forschungsgemeinschaft Graduiertenkolleg $1150$ ``Homotopy and Cohomology''.

\section{Symmetric spectra}  \label{sec:def}
\subsection{Definition} We begin by recalling the definition of a symmetric spectrum. In this paper a topological space is always assumed to be compactly-generated and weak Hausdorff, so that the category of spaces becomes closed symmetric monoidal with respect to the cartesian product. We let $S^n$ denote the $n$-sphere, by which we mean the one-point compactification of $\R^n$ in the topological case and the $n$-fold smash product of $S^1\defeq \underline{\Delta}^1/\partial \underline{\Delta}^1$ in the simplicial case.
\begin{Def}[Symmetric spectrum]
A \emph{symmetric spectrum} $X$ of spaces or simplicial sets consists of \begin{itemize}
\item a based $\Sigma_n$-space/$\Sigma_n$-simplicial set $X_n$ and
\item a based \emph{structure map} $\sigma_n:X_n\wedge S^1 \to X_{n+1}$
\end{itemize}
for all $n\in \N$. This data has to satisfy the condition that for all $n,m\in\mathbb{N}$ the \emph{iterated structure map} \[
	\sigma_n^m: X_n\wedge S^m\cong (X_n\wedge S^1)\wedge S^{m-1}\xr{\sigma_n\wedge S^{m-1}} X_{n+1}\wedge S^{m-1}\xr{\sigma_{n+1}\wedge S^{m-2}} \hdots \xr{\sigma_{n+m-1}}  X_{n+m}
\]
is $(\Sigma_n\times \Sigma_m)$-equivariant, with $\Sigma_m$ acting on $S^m$ by permuting the coordinates.

A \emph{morphism of symmetric spectra} $f:X\to Y$ is a sequence of based $\Sigma_n$-equivariant maps $f_n:X_n\to Y_n$ such that $f_{n+1}\circ \sigma_n^{(X)}=\sigma_{n+1}^{(Y)}\circ (f_n\wedge S^1)$ for all $n\in \N$.
\end{Def}

We denote the category of symmetric spectra over spaces or simplicial sets by $Sp^\Sigma_{\T}$ and $Sp^\Sigma_{\mathcal{S}}$ respectively. In statements that make sense in both categories we sometimes simply write $Sp^{\Sigma}$ and mean they hold in either of them.

\begin{Example}[Suspension spectra] Every based space or simplicial set $A$ gives rise to a suspension symmetric spectrum $\Sigma^{\infty}A$ whose $n$-th level is $A\wedge S^n$ with $\Sigma_n$-action through $S^n$ and structure map the associativity isomorphism $(A\wedge S^n)\wedge S^1\cong A\wedge S^{n+1}$. For $A=S^0$ this gives the sphere spectrum~$\mathbb{S}$.
\end{Example}

\begin{Remark}[$G$-symmetric spectra] Throughout this paper we will often make use of the theory of $G$-symmetric spectra for a fixed finite group $G$, by which we simply mean a symmetric spectrum with a $G$-action.
\end{Remark}

\subsection{Evaluations} \label{sec:evaluation}
Let $G$ be a finite group, $M$ a finite $G$-set of cardinality $m$. We denote by $\Bij(\underline{m},M)$ the discrete space or simplicial set of bijections between the sets $\underline{m}=\{1,\hdots,m\}$ and $M$. It possesses a right $\Sigma_m$-action by precomposition and a left $G$-action by postcomposition with the action on $M$.
\begin{Def}[Evaluation] The \emph{evaluation} of a symmetric spectrum $X$ \emph{on $M$} is defined as
\begin{align*}	X(M) & \defeq X_m \wedge_{\Sigma_m} \Bij(\underline{m},M)_+ \\ & \defeq X_m \wedge \Bij(\underline{m},M)_+/((\sigma x,f)\sim(x,f\sigma),\sigma \in \Sigma_m),
 \end{align*}
with $G$-action through $M$.
\end{Def}

\begin{Remark} This is the special case of an evaluation of a $G$-symmetric spectrum on a finite $G$-set, in which case $G$ acts diagonally on $X(M)=X_m \wedge_{\Sigma_m} \Bij(\underline{m},M)_+$.
\end{Remark}

The following are two examples of evaluations:

\begin{Example} Let $A$ be a based space or simplicial set, $M$ a finite $G$-set. We denote by $S^M$ the smash product of $M$ copies of $S^1$ with permutation $G$-action, generalizing the definition of the $\Sigma_n$-permutation sphere $S^n$. Then the map $(\Sigma^\infty A)(M)\to A\wedge S^M$ that sends a class $[(a\wedge x)\wedge f]$ to $a\wedge f_*(x)$ is a $G$-isomorphism.
\end{Example}
\begin{Example} \label{exa:nontriveva} Let $G$ be the symmetric group $\Sigma_n$ and $M$ be the natural $\Sigma_n$-set $\underline{n}$, $X$ a symmetric spectrum. Then $X(\underline{n})$ is canonically isomorphic to $X_n$ with the $\Sigma_n$-action that is part of the data of the symmetric spectrum $X$. In contrast, evaluating at $\{1,\hdots,n\}$ with \emph{trivial} $\Sigma_n$-action yields $X_n$ with trivial action.
\end{Example}

Moreover, these evaluations are connected by so-called generalized structure maps: Let $G$ be a finite group, $M$ and $N$ two finite $G$-sets of cardinalities $m$ and $n$, respectively, and $X$ a symmetric spectrum. We further choose a bijection $\psi: \underline{n} \xr{\cong}N$.

\begin{Def}[Generalized structure map] The map \begin{eqnarray*}
	 \sigma_M^N: X(M)\wedge S^N & \to & X(M\sqcup N) \\
	 ([x\wedge f]\wedge s) & \mapsto & [\sigma_m^n(x\wedge \psi ^{-1}_*(s))\wedge (f\sqcup \psi)]
\end{eqnarray*}
is called the \emph{generalized structure map} of $M$ and $N$.
\end{Def}
It is straightforward to check that the generalized structure map does not depend on the choice of bijection $\psi: \underline{n} \xr{\cong}N$. Furthermore, it is $G$-equivariant for the diagonal $G$-action on $X(M)\wedge S^N$. Again this is a special case of generalized structure maps for $G$-symmetric spectra.

\subsection{Global free symmetric spectra}
\label{sec:free}
For every finite group $G$ and every finite $G$-set $M$ the above construction yields functors  \[-(M):Sp^{\Sigma}_{\T}\to G\T_* \hspace{1cm}\text{ and} \hspace{1cm} -(M):Sp^{\Sigma}_{\mathcal{S}} \to G\mathcal{S}_*. \]
These functors have left adjoints $F_M^G$, which is a consequence of the existence of a left adjoint for the analogous evaluation functor from $G$-symmetric spectra to based $G$-spaces.

Here we only give the necessary definitions to construct them, more details can be found in \cite[Sec. 2.7]{Hau14}. Given a finite $K$-set $N$ for another finite group $K$ we put
\[\mathbf{\Sigma}(M,N)\defeq \bigvee \limits_{\alpha:M\hookrightarrow N\text{ injective}}^{}S^{N-\alpha(M)}. \]
This based space/simplicial set carries a right $G$-action by precomposition on the indexing wedge and a commuting left $K$-action for which an element $k$ sends a pair $(\alpha,x\in S^{N-\alpha(M)})$ to the pair $(k\circ \alpha,k\cdot x\in S^{N-k\cdot \alpha(M)})$. Given another finite $K$-set $N'$ there is a natural $(G^{op}\times K)$-equivariant map \[ \sigma_N^{N'}:\mathbf{\Sigma}(M,N)\wedge S^{N'}\to \mathbf{\Sigma}(M,N\sqcup N'); (\alpha,x)\wedge y\mapsto (\alpha,x\wedge y). \]

\begin{Def} Let $A$ be a based $G$-space/$G$-simplicial set and $M$ a finite $G$-set. Then the \emph{global free symmetric spectrum on $A$ in level $M$} is defined as $(F_M^G(A))_n\defeq A\wedge_G \mathbf{\Sigma}(M,\underline{n})$ with structure map \[ A\wedge_G \sigma_{\underline{n}}^1:(A\wedge_G \mathbf{\Sigma}(M,\underline{n}))\wedge S^1\to A\wedge_G \mathbf{\Sigma}(M,\underline{n+1}). \]
\end{Def}
More generally, if $N$ is a finite $K$-set, the evaluation $(F_M^G(A))(N)$ is canonically isomorphic to $A\wedge_G \mathbf{\Sigma}(M,N)$ with $K$-action through $N$. The generalized structure maps arise by smashing $\sigma_N^{N'}$ with $A\wedge_G -$. Then we have:
\begin{Prop}Let $M$ be a finite $G$-set, $A$ a based $G$-space/$G$-simplicial set and $X$ a symmetric spectrum. Then the assignment
\begin{eqnarray*}\map_{Sp^{\Sigma}}(F_M^G(A),X)& \to & \map_G(A,X(M))\\ (f:F_M^G(A)\to X) & \mapsto & (A\xr{[-\wedge \{id_M\}_+]}A\wedge_G\Sigma(M,M)\xr{f(M)} X(M)) \end{eqnarray*}
is a natural isomorphism.
\end{Prop}
Here, the expression $\map_{Sp^{\Sigma}}(-,-)$ refers to the space/simplicial set of morphisms between two symmetric spectra, which is recalled in the following paragraph.
\begin{proof} This follows from \cite[Proposition 2.20]{Hau14}, since by definition $F_M^G(A)$ is the $G$-quotient of the free $G$-symmetric spectrum $\mathscr{F}_M(A)$ and we are mapping into spectra with trivial $G$-action.
\end{proof}

\subsection{Mapping spaces and spectra, smash products and shifts}
In this section we quickly recall various point-set constructions for symmetric spectra, which are all introduced in \cite{HSS00}.

\begin{Example}[(Co-)Tensoring over based spaces]
Every based space/simplicial set $A$ gives rise to a functor $A\wedge -:Sp^{\Sigma}\to Sp^{\Sigma}$ by smashing each level and structure map with $A$. It is left adjoint to $\map(A,-):Sp^{\Sigma}\to Sp^{\Sigma}$, defined via $\map(A,X)_n=\map(A,X_n)$ with structure maps adjoint to $\map(A,X_n)\xr{\widetilde{\sigma}_n} \map(A,\Omega (X_{n+1}))\cong \Omega (\map(A,X_{n+1}))$.
\end{Example}
\begin{Example}[Geometric realization] Symmetric spectra of spaces and simplicial sets are related by the adjunction of geometric realization $|.|$ and singular complex $\mathcal{S}$. Both functors are constructed by applying the space level version levelwise, making use of the fact that $|.|$ commutes with $-\wedge S^1$ and $\mathcal{S}$ commutes with $\Omega(-)$ to obtain structure maps (similarly to the previous example). 
\end{Example}

\begin{Example}[Shifts] For every natural number $n$ there is an endofunctor $\sh^n:Sp^{\Sigma}\to Sp^{\Sigma}$ defined by $\sh^n(X)_m\defeq X_{n+m}$ with $\Sigma_m$-action through the last $m$ coordinates and structure maps shifted by~$n$. There is a natural transformation $\alpha_X^n:S^n\wedge X\to \sh^n (X)$ given in level $m$ by the composite \[ S^n\wedge X_m\cong X_m\wedge S^n\xr{\sigma_n^m} X_{m+n}\xr{X(\tau_{m,n})}X_{n+m}={\sh}^n(X)_m,\]
where $\tau_{m,n}$ denotes the permutation in $\Sigma_{m+n}$ that moves the first $m$ elements $\{1,\hdots,m\}$ past the last $n$ elements $\{m+1,\hdots,m+n\}$ and preserves the order of both of these subsets.

In fact, via the same formula one can shift along arbitrary finite $G$-sets $M$, but the result $\sh^M(X)$ is in general a $G$-symmetric spectrum with non-trivial $G$-action.
\end{Example}

\begin{Example}[Mapping spaces] Given two symmetric spectra of spaces $X$ and $Y$, the set of all symmetric spectra morphisms $f:X\to Y$ carries a natural topology as a subspace of the infinite product of the levelwise mapping spaces $\map(X_n,Y_n)$, yielding a mapping space $\map_{Sp^{\Sigma}}(X,Y)$. If $X$ and $Y$ are symmetric spectra of simplicial sets, there is a mapping simplicial set $\map_{Sp^{\Sigma}}(X,Y)$ whose $n$-simplices are given by the set of symmetric spectra morphisms from $\underline{\Delta}^n_+\wedge X$ to $Y$. 
\end{Example}

\begin{Example}[Internal Hom] Combining this with the shifts above gives internal homomorphism spectra $\Hom(X,Y)$ defined by $\Hom(X,Y)_n\defeq \map_{Sp^{\Sigma}}(X,\sh^nY)$ with $\Sigma_n$-action through the first $n$ coordinates in $\sh^n(Y)$ and structure map sending a pair $(f:X\to \sh^n(Y),x\in S^1)$ to the composite  \[X\xr{x\wedge f} S^1\wedge {\sh}^n(Y)\xr{\alpha_{{\sh}^n(Y)}^1} {\sh}^{n+1}(Y).\]
\end{Example}

\begin{Example}[Smash product] As shown in \cite[Sec. 2]{HSS00}, the category of symmetric spectra carries a symmetric monoidal smash product $\wedge$ with unit $\mathbb{S}$, uniquely characterized by the fact that $-\wedge X$ is left adjoint to $\Hom(X,-)$.
\end{Example}
\section{Model structures} \label{sec:model}
In this section we construct global model structures on the category of symmetric spectra, beginning with a level model structure which is later left Bousfield localized to obtain a stable version.
\subsection{Level model structure}
We recall the standard model structure on equivariant spaces and simplicial sets:
\begin{Def} A map $f:A\to B$ of based $G$-spaces (or based $G$-simplicial sets) is called a
\begin{itemize}
 \item \emph{$G$-weak equivalence} if the map $f^H:X^H\to Y^H$ is a weak equivalence for all subgroups $H$ of~$G$.
 \item \emph{$G$-fibration} if the map $f^H:X^H\to Y^H$ is a Serre fibration (resp. Kan fibration) for all subgroups $H$ of $G$.
 \item \emph{$G$-cofibration} if it has the left lifting property with respect to all maps that are simultaneously $G$-weak equivalences and $G$-fibrations.
\end{itemize}
\end{Def}
\begin{Remark} A map of based $G$-simplicial sets is a $G$-cofibration if and only if it is degreewise injective.
\end{Remark}
It is well-known that the above classes assemble to a proper, cofibrantly generated and monoidal model structure on the category of based $G$-spaces (based $G$-simplicial sets). We make use of it to construct a global level model structure on symmetric spectra:

\begin{Def} \label{def:levmod} A morphism $f:X\to Y$ of symmetric spectra is called a
\begin{itemize}
 \item \emph{global level equivalence} if each level $f_n:X_n\to Y_n$ is a $\Sigma_n$-weak equivalence.
 \item \emph{global level fibration} if each level $f_n:X_n\to Y_n$ is a $\Sigma_n$-fibration.
 \item \emph{flat cofibration} if each latching map $\nu_n[f]:X_n\cup_{L_n (X)} L_n(Y)\to Y_n$ is a $\Sigma_n$-cofibration.
\end{itemize}
\end{Def}
For the definition of latching spaces and maps we refer to \cite[Def. 5.2.1]{HSS00} or \cite[Sec. 2.8]{Hau14}. The following gives a different interpretation of global level equivalences and fibrations:
\begin{Lemma} A morphism $f:X\to Y$ of symmetric spectra is a global level equivalence (global level fibration) if and only if for all finite groups $G$ and all finite $G$-sets $M$ the map $f(M)^G:X(M)^G\to Y(M)^G$ is a weak equivalence (resp. Serre/Kan fibration) on $G$-fixed points. 
\end{Lemma}
\begin{proof} Given a finite $G$-set $M$, any choice of bijection $\underline{m}\cong M$ defines a homomorphism $\varphi:G\to \Sigma_m$ and the $G$-fixed points $X(M)^G$ are naturally identified with $X_m^{\varphi(G)}$. This translates between the different formulations.
\end{proof}

\begin{Remark} \label{rem:gleveq} In \cite{Hau14}, a morphism $f:X\to Y$ of $G$-symmetric spectra is a $G$-level equivalence if for all subgroups $H$ of $G$ and all finite $H$-sets $M$ the map $f(M)^H:X(M)^H\to Y(M)^H$ is a weak equivalence. Hence, a morphism of symmetric spectra is a global level equivalence if and only if it is a $G$-level equivalence for all finite groups $G$. Furthermore, every flat cofibration of symmetric spectra is a $G$-flat cofibration of $G$-symmetric spectra when given the trivial $G$-action.
\end{Remark}

\begin{Prop}[Level model structure] \label{prop:levmod} The global level equivalences, global level fibrations and flat cofibrations define a proper,  cofibrantly generated and monoidal model structure on the category of symmetric spectra, called the \emph{global level model structure}.
\end{Prop}
\begin{proof} The existence of the model structure and its properness follows from \cite[Prop. 2.30]{Hau14} for $G$ the trivial group, since the strong consistency condition (\cite[Def. 2.29]{Hau14}) is satisfied. Monoidality is a consequence of \cite[Cor. 2.39]{Hau14} for each finite group separately.
\end{proof}
Since the suspension spectrum functor from based spaces (simplicial sets) is a strong monoidal left Quillen functor, the monoidality of the global model structure in particular implies that it is topological (respectively simplicial). Let $I$ and $J$ denote sets of generating cofibrations and acyclic cofibrations, respectively, for the Quillen model structure on topological spaces or simplicial sets. Then sets of generating (acyclic) cofibrations for the global level model structure are given by \[ I_{gl}^{lev}=\{F_{\underline{n}}^H(i)\ |\ n\in \N, H\leq \Sigma_n,i\in I\}\]  and\[ J_{gl}^{lev}=\{F_{\underline{n}}^H(j)\ |\ n\in \N, H\leq \Sigma_n,j\in J\}, \]
respectively, where in each case the maps $i$ and $j$ are thought of as maps of $H$-spaces with trivial action and $H$ acts on $\underline{n}$ via its embedding into $\Sigma_n$.

In order to obtain a global model structure on commutative symmetric ring spectra we will also need a positive version of the global level model structure. For this we call a morphism $f:X\to Y$ a \emph{positive global level equivalence} (resp. \emph{positive global level fibration}) if $f_n:X_n\to Y_n$ is a $\Sigma_n$-weak equivalence (resp. $\Sigma_n$-fibration) for all $n\geq 1$. Furthermore, a \emph{positive flat cofibration} is a flat cofibration which is an isomorphism in degree $0$. Then we have:
\begin{Prop}[Positive level model structure] The positive global level equivalences, positive global level fibrations and positive flat cofibrations define a proper and cofibrantly generated model structure on the category of symmetric spectra, called the \emph{positive global level model structure}.
\end{Prop}
\begin{proof} As above, this model structure can be obtained via \cite[Prop. 2.30]{Hau14}. 
\end{proof}
The positive global level model structure satisfies the pushout product axiom but not the unit axiom, so it is not quite monoidal.

\subsection{Global equivalences}
In order to define the global (stable) equivalences we have to recall the notions of $G\Omega$-spectrum and $G$-stable equivalence for a fixed finite group $G$. In comparing to \cite{Hau14}, we always use the notions formed with respect to a complete $G$-set universe $\U_G$. These notions do not depend on a particular choice of such and so we omit it from the notation.
\begin{Def}[$G\Omega$-spectra] \label{def:gomega} A $G$-symmetric spectrum $X$ is called a \emph{$G\Omega$-spectrum} if for all subgroups $H$ of $G$ and all finite $H$-sets $M$ and $N$ the adjoint structure map
\[ \widetilde{\alpha}_M^N:X(M)\to \Omega^N X(M\sqcup N) \]
is an $H$-weak equivalence.
\end{Def}
In the simplicial case the expression $\Omega^N$ above is to be understood as the derived mapping space, i.e., as $\Omega^N (\mathcal{S}(|-|))$. As recalled in Remark \ref{rem:gleveq}, a map $f:X\to Y$ of $G$-symmetric spectra is a \emph{$G$-level equivalence} if for all subgroups $H\leq G$ and all finite $H$-sets $M$ the evaluation $f(M)^H:X(M)^H\to Y(M)^H$ is a weak equivalence. We denote the localization of $G$-symmetric spectra at the $G$-level equivalences by $\gamma_G:GSp^{\Sigma}\to GSp^{\Sigma}[G\text{-level eq.}^{-1}]$.
\begin{Def}[$G$-stable equivalence] \label{def:gstable} A morphism $f:X\to Y$ of $G$-symmetric spectra is a \emph{$G$-stable equivalence} if for all $G\Omega$-spectra $Z$ the map
	\[ GSp^{\Sigma}[G\text{-level eq.}^{-1}](Y,Z)\xr{\gamma_G(f)^*}GSp^{\Sigma}[G\text{-level eq.}^{-1}](X,Z) \]
is a bijection.
\end{Def}

Now we can define:
\begin{Def}[Global equivalence] \label{def:globalequivalence}A morphism $f:X\to Y$ of symmetric spectra is a \emph{global equivalence} if it is a $G$-stable equivalence for all finite groups $G$ when given the trivial $G$-action.
\end{Def}
\begin{Example} \label{exa:levelisstable} Every global level equivalence is a global equivalence, since it is a $G$-level equivalence for all finite groups $G$. In fact, every ``eventual level equivalence'' $f:X\to Y$ (in the sense that for every finite group $G$ there exists a finite $G$-set $M$ such that $f(M\sqcup N)^G:X(M\sqcup N)^G\to Y(M\sqcup N)^G$ is a weak equivalence for all finite $G$-sets $N$) is a global equivalence. This is easiest to see via Proposition~\ref{prop:piisoglobal}, since every eventual level equivalence induces an isomorphism on equivariant homotopy groups, which are discussed in Section \ref{sec:homgroups}. \end{Example}

We make the definition of a global equivalence more concrete and consider the (trivial $G$-action/$G$-fixed points) adjunction
\[ {\triv}_G:Sp^{\Sigma}\rightleftarrows GSp^{\Sigma}:(-)^{G}. \]
By definition, a map $f$ of symmetric spectra is a global equivalence if and only if $\triv_G(f)$ is a $G$-stable equivalence for all $G$. Using the global level model structure on $Sp^{\Sigma}$ and the $G$-flat level model structure on $GSp^{\Sigma}$, the adjunction forms a Quillen pair (since the trivial action functor preserves all cofibrations and weak equivalences) and so it can be derived to an adjunction between the homotopy categories
\[ \mathbb{L}{\triv}_G:Sp^{\Sigma}[\text{global level eq.}^{-1}]\rightleftarrows GSp^{\Sigma}[G\text{-level eq.}^{-1}]:(-)^{\mathbb{R}G}, \]
where the trivial action functor does not really need to be derived as it is homotopical. Using this adjunction and the definition of a $G$-stable equivalence we see:
\begin{Cor} A map $f:X\to Y$ of symmetric spectra is a global equivalence if and only if for all finite groups $G$ and all $G\Omega$-spectra $Z$ the map
\[ Sp^{\Sigma}[\textnormal{global level eq.}^{-1}](Y,Z^{\R G})\xr{\gamma(f)^*}Sp^{\Sigma}[\textnormal{global level eq.}^{-1}](X,Z^{\R G}) \]
is a bijection.
\end{Cor}
Here, $\gamma:Sp^{\Sigma}\to Sp^{\Sigma}[\text{global level eq.}^{-1}]$ denotes the localization functor. This may still be unsatisfactory, because the definition is not intrinsic to symmetric spectra as it is not clear which symmetric spectra arise as the derived fixed points of $G{\Omega}$-spectra. It turns out that these fixed points are again equivariant $\Omega$-spectra, in the following global sense:
\begin{Def}[Global $\Omega$-spectra] \label{def:globalomega} A symmetric spectrum $X$ is called a \emph{global $\Omega$-spectrum} if for all finite groups $G$ and all finite $G$-sets $M$ and $N$ of which $M$ is faithful the adjoint generalized structure map
\[ \widetilde{\sigma}_M^N:X(M)\to \Omega^N (X(M\sqcup N)) \]
is a $G$-weak equivalence.
\end{Def}
Again, the mapping space has to be derived in the simplicial case. In particular, every global $\Omega$-spectrum is a non-equivariant $\Omega$-spectrum. In general, a global $\Omega$-spectrum $X$ is not quite a $G\Omega$-spectrum for non-trivial finite groups $G$, as there is no faithfulness condition in Definition \ref{def:gomega}. However, every faithful finite $G$-set $N$ gives rise to a $G\Omega$-replacement $X\to \Omega^N (\sh^N (X))$ of $X$ (up to eventual $G$-level equivalence), but $\Omega^N (\sh^N (X))$ has non-trivial $G$-action and is thus not an object of the global category.
It is usually not possible to replace a symmetric spectrum by a globally equivalent symmetric spectrum which is a $G\Omega$-spectrum for all finite groups $G$ at once (the most prominent exception being the Eilenberg-MacLane spectrum $\HEM\mathbb{Z}$ for the constant global functor $\underline{\mathbb{Z}}$ discussed in \cite[Constr. V.3.21]{Sch15}).

As promised, we have:
\begin{Prop} \label{prop:fixedgomega} The derived fixed points $X^{\R G}$ of a $G\Omega$-spectrum $X$ form a global $\Omega$-spectrum. 
\end{Prop}
\begin{proof}As remarked above, we can use a $G$-flat fibrant replacement $X^f$ of $X$ to compute its right derived fixed points. We now recall from \cite[Sec. 2.9]{Hau14} what it means for a $G$-symmetric spectrum to be $G$-flat fibrant. Given two groups $G$ and $K$ we let $\mathcal{F}^{G,K}$ denote the family of subgroups of $G\times K$ whose intersection with $\{e\}\times K$ is trivial. Every such subgroup is of the form $\{(h,\varphi(h))\ |\ h\in H\}$ for a unique subgroup $H$ of $G$ and group homomorphism $\varphi:H\to K$. Then the fact that $X^f$ is $G$-flat fibrant means that each level $X^f_n$ is $(G\times \Sigma_n)$-fibrant and in addition cofree with respect to the family $\mathcal{F}^{G,\Sigma_n}$, i.e., the map $X^f_n\to \map(E\mathcal{F}^{G,\Sigma_n}_+,X^f_n)$ is a $(G\times \Sigma_n$)-weak equivalence, where $E\mathcal{F}^{G,\Sigma_n}$ is a universal space for $\mathcal{F}^{G,\Sigma_n}$ (cf. \cite[Defs. 2.25 and 1.20]{Hau14}).

We now show that $(X^f)^G$ forms a global $\Omega$-spectrum. Let $K$ be a finite group, $M$ and $N$ be finite $K$-sets of which $M$ faithful (and of cardinality $m$). We consider the evaluation $X^f(M)=X^f_m\wedge_{\Sigma_m} (\Bij(\underline{m},M)_+)$ and give it a $(G\times K)$-action by letting $G$ act through $X^f_m$ and $K$ through $M$. Likewise, we obtain a $(G\times K)$-action on $X^f(M\sqcup N)$ and hence also on $\Omega^N (X^f(M\sqcup N))$.

We claim the following:
\begin{enumerate}[(i)]
 \item The map $\widetilde{\sigma}_M^N:X^f(M)\to \Omega^N(X^f(M\sqcup N))$ is an $\mathcal{F}^{G,K}$-weak equivalence, i.e., it induces a weak equivalence on all fixed points for subgroups in the family $\mathcal{F}^{G,K}$.
 \item Both $X^f(M)$ and $\Omega^N(X^f(M\sqcup N))$ are $\mathcal{F}^{G,K}$-cofree.
\end{enumerate}
Together these imply that $\widetilde{\sigma}_M^N:X^f(M)\to \Omega^N(X^f(M\sqcup N))$ is a $(G\times K)$-weak equivalence, as every $\F^{G,K}$-weak equivalence between $\F^{G,K}$-cofree $(G\times K)$-spaces is a $(G\times K)$-weak equivalence. In particular, the induced map on $G$-fixed points $(\widetilde{\sigma}_M^N)^G:(X^f)^G(M)\to \Omega^N ((X^f)^G(M\sqcup N))$ is a $K$-weak equivalence, which proves the proposition. 

Hence it remains to show the claims, we begin with the first one. We let $H$ be a subgroup of $G$ and $\varphi:H\to K$ a group homomorphism. Then the composite $H\to K\to \Sigma_M$ defines an $H$-action on $M$ (and likewise on $N$), which we denote by $\varphi^*(M)$. Pulling back $X^f(M)$ and $X^f(M\sqcup N)$ along the graph of $\varphi$ yields the $H$-spaces $X^f(\varphi^*(M))$ and $X^f(\varphi^*(M\sqcup N))$. In other words, we have to check whether the adjoint structure map $\widetilde{\sigma}_m^n:X^f(\varphi^*(M))\to \Omega^{\varphi^*(N)}(X^f(\varphi^*(M\sqcup N)))$ induces a weak equivalence on $H$-fixed points, but this is the case since $X^f$ is a $G\Omega$-spectrum.

The second claim follows from the observation that when restricting $E\mathcal{F}^{G,\Sigma_m}$ along $id\times \psi$ for an injective group homomorphism $\psi:K\to \Sigma_m$ one obtains a model for $E\mathcal{F}^{G,K}$. This finishes the proof.
\end{proof}
It will be a consequence of Theorem \ref{theo:stablemod} that global $\Omega$-spectra are precisely the local objects with respect to the class of global equivalences. In other words, one could alternatively characterize global equivalences as those morphisms that induce bijections on all morphism sets into global $\Omega$-spectra in the global level homotopy category. 

\subsection{Stable model structure} \label{sec:stablemod} In this section we introduce the global stable model structure on symmetric spectra. We begin by constructing a global $\Omega$-spectrum replacement functor up to natural global equivalence.

For this we let $G$ be a finite group, $M$ and $N$ two finite $G$-sets and define $\lambda_M^N:F^G_{M\sqcup N}(S^N)\to F_M^G(S^0)$ to be adjoint to the embedding $S^N\hookrightarrow \mathbf{\Sigma}(M,M\sqcup N)/G=(F_M^G(S^0))(M\sqcup N)$ associated to the inclusion $M\hookrightarrow M\sqcup N$ (cf. Section \ref{sec:free} for the definition of $\mathbf{\Sigma}(-,-)$ and global free symmetric spectra). Under the adjunction isomorphism, $\lambda_M^N$ represents the adjoint generalized structure map on $G$-fixed points
\[ \map_{Sp^{\Sigma}}(F^G_M (S^0),X)\cong X(M)^G\xr{(\widetilde{\sigma}_M^N)^G} (\Omega^N X(M\sqcup N))^G \cong \map_{Sp^{\Sigma}}(F^G_{M\sqcup N}(S^N),X). \]
The morphisms $\lambda_M^N$ are usually not cofibrations, so we factor them as
\[ F^G_{M\sqcup N}(S^N)\xr{\overline{\lambda}_M^N} \Cyl(\lambda_M^N)\xr{r_M^N} F_M^G(S^0)\]
via the levelwise mapping cylinder $\Cyl(-)$. It is a formal consequence, as explained in the proof of \cite[Lemma 3.4.10]{HSS00},  that $\overline{\lambda}_M^N$ is a flat cofibration, since the global level model structure is topological/simplicial. Finally, we define 
\[ J_{gl}^{st} =\{ i\square \overline{\lambda}_M^N\ |\ i\in I,G\text{ finite}, M,N\text{ finite $G$-sets with $M$ faithful} \}\cup J_{gl}^{lev},
\]
where $I$ is a set of generating cofibrations of the Quillen model structure on based spaces/simplicial sets. The notation $f\square g$ stands for the pushout product $(A\wedge Y)\cup_{A\wedge X}(B\wedge X)\to (B\wedge Y)$ of a map $f:A\to B$ of based spaces/simplicial sets with a morphism $g:X\to Y$ of symmetric spectra.
More precisely, we only include $i\square \overline{\lambda}_M^N$ for a chosen system of representatives of isomorphism classes of triples $(G,M,N)$ to ensure that $J_{gl}^{st}$ is a set. Then we have:
\begin{Prop} \label{prop:omegalifting} For a symmetric spectrum $X$ the following are equivalent:
\begin{itemize}
 \item $X$ is a level fibrant global $\Omega$-spectrum.
 \item $X$ has the right lifting property with respect to the set $J_{gl}^{st}$.
\end{itemize}
\end{Prop}
\begin{proof} We already know that $X$ is global level fibrant if and only if it has the right lifting property with respect to $J_{gl}^{st}$. By adjunction, $X$ has the right lifting property with respect to $\{i\square \overline{\lambda}_M^N\}_{i\in I}$ if and only if 
\[ \map_{Sp^\Sigma}(\overline{\lambda}_M^N,X): \map_{Sp^\Sigma}(\Cyl(\lambda_M^N),X)\to \map_{Sp^\Sigma}(F^G_{M\sqcup N}(S^N),X) \]
has the right lifting property with respect to the set $I$. Since the global level model structure is topological/simplicial, this map is always a Serre/Kan fibration. Hence, it has the right lifting property with respect to $I$ if and only if it is a weak homotopy equivalence. Since $r_M^N$ is a homotopy equivalence of symmetric spectra, this in turn is equivalent to
\[ \map_{Sp^\Sigma}(F^G_M (S^0),X)\xr{\map_{Sp^\Sigma}(\lambda_M^N,X)} \map_{Sp^\Sigma}(F^G_{M\sqcup N}(S^N),X) \]
being a weak homotopy equivalence. As remarked above, this map can be identified with the $G$-fixed points of the adjoint generalized structure map $\widetilde{\sigma}_M^N$ of $X$, which finishes the proof.
\end{proof}

\begin{Cor} If $M$ is faithful, then $\lambda_M^N$ is a global equivalence.
\end{Cor}
\begin{proof} This is a consequence of Propositions \ref{prop:fixedgomega} and \ref{prop:omegalifting} and the fact that $F_{M\sqcup N}^G (S^N)$ and $F_M^G (S^0)$ are flat.
\end{proof}

Since the global level model structure is topological/simplicial, it follows that every morphism in $J^{st}_{gl}$ is a flat cofibration. Furthermore, all domains and codomains of morphisms in $J^{st}_{gl}$ are small with respect to countably infinite sequences of flat cofibrations. So we can apply the small object argument (cf. \cite[Sec. 7.12]{DS95}) to obtain a functor $Q:Sp^{\Sigma}\to Sp^{\Sigma}$ with image in global $\Omega$-spectra and a natural relative $J^{st}_{gl}$-cell complex $q:\id\to Q$.
Since every morphism in $J^{st}_{gl}$ is a flat cofibration and global equivalence, it follows from \cite[Props. 4.2, 4.4 and 4.5]{Hau14} applied to each finite group separately that every relative $J^{st}_{gl}$-cell complex is a global equivalence. In particular, the morphisms $q_X:X\to QX$ are always global equivalences. This also implies that $Q$ preserves global equivalences by 2-out-of-3. Before we use these properties to construct the global stable model structure we need one more lemma:

\begin{Lemma} \label{lem:leveqomega} Every global equivalence between global $\Omega$-spectra is a global level equivalence. 
\end{Lemma}
\begin{proof} Let $f:X\to Y$ be a global equivalence of global $\Omega$-spectra. We have to show that each $f_n$ is a $\Sigma_n$-weak equivalence. For this we again denote by $\underline{n}$ the tautological $\Sigma_n$-set and consider the following diagram of $\Sigma_n$-symmetric spectra:
\[  \xymatrix{ X \ar[r]^-{\alpha_X^{\underline{n}}} \ar[d]_f & \Omega^{\underline{n}} (sh^{\underline{n}} X) \ar[d]^{\Omega^{\underline{n}} (sh^{\underline{n}} f)}\\
	Y \ar[r]_-{\alpha_Y^{\underline{n}}} & \Omega^{\underline{n}} (sh^{\underline{n}} Y) } 
\]
Since $X$ and $Y$ are global $\Omega$-spectra the horizontal arrows $\alpha_X^{\underline{n}}$ and $\alpha_Y^{\underline{n}}$ induce $\Sigma_n$-weak equivalences on all evaluations at faithful $\Sigma_n$-sets. In particular, using Example \ref{exa:levelisstable} we see that they are both $\Sigma_n$-stable equivalences and so $\Omega^{\underline{n}} (sh^{\underline{n}} f)$ is also a $\Sigma_n$-stable equivalence. Furthermore, since $\underline{n}$ is a faithful $\Sigma_n$-set, the $\Sigma_n$-symmetric spectra $\Omega^{\underline{n}} (sh^{\underline{n}} X)$ and $\Omega^{\underline{n}} (sh^{\underline{n}} Y)$ are $\Sigma_n\Omega$-spectra. This implies that $\Omega^{\underline{n}} (sh^{\underline{n}} f)$ is even a $\Sigma_n$-level equivalence by the Yoneda lemma. In particular, it induces a $\Sigma_n$-weak equivalence when evaluated on $\underline{n}$ and hence so does $f$ (again using that the horizontal arrows induce $\Sigma_n$-weak equivalences on all faithful evaluations). This finishes the proof.
\end{proof}
Finally, a morphism of symmetric spectra is called a \emph{(positive) global fibration} if it has the right lifting property with respect to all morphisms that are (positive) flat cofibrations and global equivalences. Then we have:
\begin{Theorem}[Global model structures] \label{theo:stablemod} The global equivalences, (positive) global fibrations and (positive) flat cofibrations define a proper, cofibrantly generated and monoidal model structure on the category of symmetric spectra, called the \emph{(positive) global stable model structure}.

Moreover, the fibrant objects of the (positive) global stable model structure are precisely the (positive) global $\Omega$-spectra.
\end{Theorem}
Here, a symmetric spectrum is called a \emph{positive global $\Omega$-spectrum} if it satisfies the condition of Definition \ref{def:globalomega} in all cases except possibly for $G=\{e\}$ and $M=\emptyset$.
\begin{proof} Both model structures are obtained via left Bousfield localization at the respective global level model structure. We apply Theorem \cite[Thm. 9.3]{Bou01}, with respect to the global $\Omega$-spectrum replacement functor $Q$ and the natural global equivalence $q:\id\to Q$ just constructed. By Lemma \ref{lem:leveqomega}, a morphism between global $\Omega$-spectra is a global equivalence if and only if it is a (positive) global level equivalence, so the global equivalences agree with the $Q$-equivalences in the sense of Bousfield's theorem.

It remains to check axioms (A1)-(A3) of \cite[Sec. 9.2]{Bou01}. Axiom (A1) requires that every (positive) global level equivalence is a global equivalence, which is Example \ref{exa:levelisstable}. For a symmetric spectrum $X$, the morphisms $q_{QX},Qq_X:QX\to QQX$ are global equivalences between global $\Omega$-spectra, hence global level equivalences by Lemma \ref{lem:leveqomega}, implying axiom (A2). For (A3) we are given a pullback square
\[
	\xymatrix{V\ar[r]^k \ar[d]_g & X \ar[d]^f\\
						W\ar[r]_h & Y}
\]
where $f$ is a (positive) global level fibration, $h$ is a global equivalence and $X$ and $Y$ are (positive) global $\Omega$-spectra. We have to show that $g$ is also a global equivalence. This is even true without any hypothesis on $X$ and $Y$, as follows by applying the dual version of \cite[Prop. 4.4]{Hau14} for every finite group $G$.

Monoidality of the model structures is again implied by the respective monoidality of the $G$-flat model structures (\cite[Prop. 4.13]{Hau14}). Finally, the statement about the fibrant objects is a consequence of the characterization of the fibrations in the localized model structure given in \cite[Thm. 9.3]{Bou01} and the fact that $X$ is a (positive) global $\Omega$-spectrum if and only if the map $q_X:X\to QX$ is a (positive) global level equivalence.
\end{proof}
The generating cofibrations are the same as for the respective level model structures. In the non-positive case, the generating acyclic cofibrations are given by $J_{gl}^{st}$, for the positive version one has to take out those maps that are not positive flat cofibrations (i.e., those involving a spectrum of the form $F^{\{e\}}_{\emptyset}(-)$). Finally, we note:

\begin{Prop} Geometric realization and singular complex form a Quillen equivalence between symmetric spectra of spaces and simplicial sets with the (positive) global model structures. 
\end{Prop}
\begin{proof} It follows directly from the unstable Quillen equivalence of spaces and simplicial sets that the adjunction is a Quillen equivalence for the global level model structures. Furthermore, a symmetric spectrum of simplicial sets $X$ is a global $\Omega$-spectrum if and only if its geometric realization is. This implies that both geometric realization and singular complex preserve and reflect global equivalences and that they form a Quillen equivalence for the global stable model structures.
\end{proof}

\section{Multiplicative properties}
\label{sec:monoidal}
We have seen in Theorem \ref{theo:stablemod} that the global model structure is monoidal, i.e., that it satisfies the pushout product and unit axiom. In this section we construct global model structures on categories of modules, algebras and commutative algebras by further checking that the monoid and strong commutative monoid axioms hold. In all cases, the properties follow directly from the respective ones for $G$-symmetric spectra.

\subsection{Model structure on module and algebra categories}
Given a model structure on symmetric spectra, a map of modules or algebras is called a weak equivalence or fibration if its underlying morphism of symmetric spectra is. We say that the given model structure lifts to the category of modules or algebras if these two classes define a model structure.
\begin{Theorem} \label{theo:modules} For every symmetric ring spectrum $R$ the positive and non-positive global stable model structures lift to the category of $R$-modules. If $R$ is commutative, these model structures are again monoidal. 
\end{Theorem}
\begin{Theorem} \label{theo:algebras} For every commutative symmetric ring spectrum $R$ the positive and non-positive global stable model structures lift to the category of $R$-algebras. Moreover, every cofibration of $R$-algebras whose source is cofibrant as an $R$-module is a cofibration of $R$-modules. 
\end{Theorem}
Both theorems are obtained via the results of \cite{SS00}, which show that it suffices to prove that the monoid axiom (stated below) holds. The main ingredient is the following:
\begin{Prop}[Flatness] \label{prop:flatness} \begin{enumerate}[(i)] \item Smashing with a flat symmetric spectrum preserves global equivalences.
    \item Smashing with an arbitrary symmetric spectrum preserves global equivalences between flat symmetric spectra. 
\end{enumerate}
\end{Prop}
\begin{proof} This is a direct consequence of \cite[Prop. 7.1]{Hau14}.
\end{proof}
For any symmetric spectrum $Y$ we denote by $\{J_{gl}^{st}\wedge Y\}_{cell}$ the class of morphisms obtained via (transfinite) compositions and pushouts from morphisms of the form $j\wedge Y$, where $j$ lies in $J_{gl}^{st}$.
\begin{Cor}[Monoid axiom] Every morphism in $\{J_{gl}^{st}\wedge Y\}_{cell}$ is a global equivalence.
\end{Cor}
\begin{proof} It suffices to consider the topological case. Every morphism of the form $j\wedge Y$ with $j\in J_{gl}^{st}$ is an $h$-cofibration (cf. \cite[Def. 4.3]{Hau14}) and a global equivalence by Proposition \ref{prop:flatness}. The class of $h$-cofibrations which are also global equivalences is closed under pushouts and composition (\cite[Prop. 4.4 and 4.5]{Hau14} for every group separately), hence every morphism in $\{J_{gl}^{st}\wedge Y\}_{cell}$ is a global equivalence. Alternatively, the corollary follows directly from the monoid axiom for the $G$-flat stable model structure on $G$-symmetric spectra (\cite[Prop. 7.3]{Hau14}). 
\end{proof}

By \cite[Thm. 4.1]{SS00}, this implies Theorems \ref{theo:modules} and \ref{theo:algebras}.

\subsection{Model structure on commutative algebra categories}
The positive global model structure also lifts to the category of commutative symmetric ring spectra (or more generally, commutative algebras over a commutative symmetric ring spectrum). We note that this is a very strong form of equivariant commutativity, which induces norm maps and power operations on equivariant homotopy groups. For this reason commutative symmetric (or orthogonal) ring spectra are called ``ultracommutative'' in \cite{Sch15} when they are considered from the point of view of global homotopy.
\begin{Theorem} \label{theo:commalgebras} For every commutative symmetric ring spectrum $R$ the positive global model structure lifts to the category of commutative $R$-algebras.

Moreover, the underlying $R$-module map of a positive flat cofibration of commutative $R$-algebras $X\to Y$ is a positive flat cofibration of $R$-modules if $X$ is (not necessarily positive) flat as an $R$-module. In particular, the symmetric spectrum underlying a positive flat commutative symmetric ring spectrum is positive flat.
\end{Theorem}
The part about positive flat cofibrations is merely a restating of Shipley's result \cite[Prop. 4.1]{Shi04}, since the cofibrations in the positive flat non-equivariant and the positive global model structure on commutative algebras are the same.

In order to prove Theorem \ref{theo:commalgebras} we make use of results of \cite{Whi14}. For this we recall that given a morphism $f:X\to Y$ of symmetric spectra, the $n$-fold pushout product $f^{\square n}$ is defined inductively via $f^{\square n}\defeq f\square f^{\square (n-1)}$.
\begin{Prop}[Strong commutative monoid axiom] Let $f:X\to Y$ be a morphism of symmetric spectra. Then:
\begin{enumerate}[(i)]
  \item If $f$ is a (positive) flat cofibration, then $f^{\square n}/\Sigma_n$ is again a (positive) flat cofibration.
  \item If $f$ is a positive flat cofibration and global equivalence, then so is $f^{\square n}/\Sigma_n$.
\end{enumerate}
\end{Prop}
\begin{proof} This follows immediately from \cite[Prop. 7.20]{Hau14}. 
\end{proof}

Applying \cite[Thm. 3.2]{Whi14} (and \cite[Prop. 4.1]{Shi04} for the part on cofibrations), we obtain Theorem \ref{theo:commalgebras}.

\section{Equivariant homotopy groups of symmetric spectra}
\label{sec:homgroups} In this section we study equivariant homotopy groups of symmetric spectra. We say that a countable $G$-set for a finite group $G$ is a \emph{complete $G$-set universe} if it allows an embedding of every finite $G$-set. Then for every symmetric spectrum $X$, every finite group $G$, every complete $G$-set universe $\U_G$ and every integer $n$ we define an abelian group $\pi_n^{G,\U_G}(X)$. Any two complete $G$-set universes are isomorphic, which will imply that $\pi_n^{G,\U_G}(X)$ only depends on the choice of $\U_G$ up to natural isomorphism. However, unlike for orthogonal spectra this isomorphism of homotopy groups is not canonical, it is affected by the choice of isomorphism of $G$-set universes. Hence, for arbitrary symmetric spectra $X$ it is misleading to simply write $\pi_n^G(X)$. This phenomenon also affects the functoriality of $\pi_n^{G,\U_G}(X)$ in group homomorphisms, which we discuss in Section \ref{sec:functoriality}.

All this is tied to the fact that equivariant homotopy groups of symmetric spectra are not homotopical, i.e., global equivalences generally do not induce isomorphisms on them. If one works with the derived versions (i.e., replacing $\pi_n^{G,\U_G}(X)$ by $\pi_n^{G,\U_G}(QX)$) these problems disappear and one obtains the same properties as for homotopy groups of orthogonal spectra. In Section \ref{sec:semistability} we discuss criteria to detect for which symmetric spectra the ``naive'' equivariant homotopy groups are already derived.

\subsection{Definition and global $\upi_*$-isomorphisms} \label{sec:defhomgroups}
Given a finite group $G$ and a complete $G$-set universe $\U_G$ we denote by $s_G(\U_G)$ the poset of finite $G$-subsets of $\U_G$, partially ordered by inclusion.
\begin{Def} Let $n\in \mathbb{Z}$ be an integer. Then the $n$-th $G$-equivariant homotopy group $\pi_n^{G,\U_G}(X)$ of a symmetric spectrum of spaces $X$ (with respect to $\U_G$) is defined as
\[ \pi_n^{G,\U_G}(X) \defeq \colim_{M\in s_G(\U)}[S^{n\sqcup M},X(M)]^G. \]
The connecting maps in the colimit system are given by the composites
\[ [S^{n\sqcup M},X(M)]^G\xr{(-)\wedge S^{N-M}} [S^{n\sqcup M\sqcup (N-M)},X(M)\wedge S^{N-M}]^G\xr{(\sigma_M^{N-M})_*}  [S^{n\sqcup N},X(N)]^G
\]
for every inclusion $M\subseteq N$. The last step implicitly uses the homeomorphism $X(M\sqcup (N-M))\cong X(N)$ induced from the canonical isomorphism $M\sqcup (N-M)\cong N$.
\end{Def}
To clarify what this exactly means for negative $n$ we choose an isometric $G$-embedding $i:\R^{\infty}\hookrightarrow (\R^{(\U_G)})^G$ and only index the colimit system over those $G$-sets $M$ in $s_G(\U)$ for which $\R^M$ contains $i(\R^{-n})$. In this case the corresponding term is given by $[S^{M-i(\R^{-n})},X(M)]^G$, the expression $M-i(\R^{-n})$ denoting the orthogonal complement of $i(\R^{-n})$ in $\R^M$. Since the space of embeddings $\R^{\infty}\hookrightarrow (\R^{(\U_G)})^G$ is contractible, the definition only depends on this choice up to \emph{canonical} isomorphism and so we leave it out of the notation. As long as $S^{n\sqcup M}$ has at least two trivial coordinates, the set $[S^{n\sqcup M},X(M)]^G$ carries a natural abelian group structure and hence so does $\pi_n^{G,\U_G}(X)$.

For a symmetric spectrum of simplicial sets we put $\pi_n^{G,\U_G}(X)\defeq \pi_n^{G,\U_G}(|X|)$.
\begin{Def} A morphism $f:X\to Y$ of symmetric spectra is called a \emph{global $\upi_*$-isomorphism} if for all finite groups $G$, all integers $n\in \mathbb{Z}$ and every complete $G$-set universe $\U_G$ the induced map $\pi_n^{G,\U_G}(f):\pi_n^{G,\U_G}(X)\to \pi_n^{G,\U_G}(Y)$ is an isomorphism.
\end{Def}
In fact it suffices to require an isomorphism for a single choice of complete $G$-set universe $\U_G$ for each finite group $G$, since any two are non-canonically isomorphic.
\begin{Remark} In the terminology of \cite[Sec. 3]{Hau14}, a morphism of symmetric spectra is a global $\upi_*$-isomorphism if and only if it is a $\upi_*^{\U_G}$-isomorphism for every finite group $G$. 
\end{Remark}

The following is immediate from the definition:
\begin{Example} Every global level equivalence is a global $\upi_*$-isomorphism. 
\end{Example}

Every global level equivalence is also a global equivalence, as we remarked in Example \ref{exa:levelisstable}. It is not obvious from the definition that this is true for arbitrary global $\upi_*$-isomorphisms, but it follows by applying Theorem \cite[Thm. 3.48]{Hau14} for each finite group $G$:
\begin{Prop} \label{prop:piisoglobal} Every global $\upi_*$-isomorphism is a global equivalence.
\end{Prop}

\subsection{Properties}
We now collect some properties of equivariant homotopy groups and global $\upi_*$-isomorphisms, all implied by their respective versions for $G$-symmetric spectra. For this we let $C(f)$ denote the levelwise mapping cone of a morphism $f:X\to Y$ of symmetric spectra, $i(f):Y\to C(f)$ the inclusion into the cone and $q(f):C(f)\to S^1\wedge X$ its cofiber. Dually, we let $H(f)$ stand for the levelwise homotopy fiber, $p(f):H(f)\to X$ the projection and $j(f):\Omega (Y)\to H(f)$ its fiber.
\begin{Prop} Let $G$ be a finite group and $\U_G$ a complete $G$-set universe. Then the following hold:
\begin{enumerate}
\item For every symmetric spectrum of spaces $X$ the unit $X\to \Omega (S^1\wedge X)$ and the counit $S^1\wedge (\Omega X)\to X$ are global $\upi_*$-isomorphisms. In particular, there are natural isomorphisms \[ \pi_{n+1}^{G,\U_G}(S^1\wedge X)\cong \pi_n^{G,\U_G}(X)\cong \pi_{n-1}^{G,\U_G}(\Omega X).\]
\item For every morphism $f:X\to Y$ of symmetric spectra of spaces the sequences
 \[ \hdots\to \pi_n^{G,\U_G}(X)\xr{f_*} \pi_n^{G,\U_G}(Y) \xr{i(f)_*} \pi_n^{G,\U_G}(C(f))\xr{q(f)_*} \pi_{n-1}^{G,\U_G} (X)\to\hdots \]
and
\[ \hdots \to \pi_{n+1}^{G,\U_G} (Y)\xr{j(f)_*} \pi_n^{G,\U_G} (H(f))\xr{p(f)_*} \pi_n^{G,\U_G} (X)\xr{f_*} \pi_n^{G,\U_G}(Y) \to \hdots \]
are exact. Furthermore, the natural morphism $S^1\wedge H(f)\to C(f)$ is a global $\upi_*$-isomorphism.
\item For every family $(X_i)_{i\in I}$ of symmetric spectra the canonical map \[ \bigoplus_{i\in I} (\pi_n^{G,\U_G}(X_i))\to \pi_n^{G,\U_G}(\bigvee_{i\in I} X_i) \] is an isomorphism of abelian groups. If $I$ is finite, the natural morphism $\bigvee_{i\in I} X_i\to \prod_{i\in I} X_i$ is a global $\upi_*$-isomorphism.
\item Smashing with a flat symmetric spectrum preserves global $\upi_*$-isomorphisms.
\end{enumerate}
\end{Prop}
In the second item we have implicitly used the isomorphisms of item ($i$) to obtain the boundary maps.
\begin{proof} These are Propositions 3.8, 3.9 and 7.1 as well as Corollaries 3.9 and 3.10 in \cite{Hau14}.
\end{proof}
This proposition also has a simplicial analog, for which in item $(i)$ and for the second long exact sequence the constructions $\Omega$ and $H(-)$ need to be replaced by their derived versions (i.e., first applying $\mathcal{S}(|-|)$).
\subsection{Functoriality}
\label{sec:functoriality}
An important feature of global homotopy theory of orthogonal spectra is that their equivariant homotopy groups enjoy a rich functoriality in the group, they form a so-called \emph{global functor}. In short, every group homomorphism $\varphi:G\to K$ induces a restriction map $\varphi^*:\pi_*^K(X)\to \pi_*^G(X)$ (depending only on its conjugacy class) and for every subgroup $H\leq G$ there is a transfer homomorphism $\tr_H^G:\pi_*^H(X)\to \pi_*^G(X)$. Moreover, restrictions and transfers are related by a double coset formula.

While the transfer homomorphism works similarly for symmetric spectra, a complication arises when one tries to construct restriction maps. To explain this, we let $X$ be a symmetric spectrum, $\varphi:G\to K$ a homomorphism of finite groups and $x\in \pi_0^{K,\U_K}(X)$ an element represented by a $K$-map $f:S^M\to X(M)$ for a finite $K$-subset $M$ of $\U_K$. Restricting all the actions along $\varphi$ and making use of the equalities $\varphi^*(S^M)= S^{\varphi^*(M)}$ and $\varphi^*(X(M))= X(\varphi^*(M))$, we can think of $f$ as a $G$-map $S^{\varphi^*(M))}\to X(\varphi^*(M))$. In order for this to represent an element $\varphi^*(x)$ in $\pi_0^{G,\U_G}(X)$ we have to choose an embedding of $\varphi^*(M)$ into $\U_G$, but such an embedding is not canonical and -- unlike for orthogonal spectra -- the outcome is in general affected by the choice one makes.
One might try to get around this by using the restricted universe $\varphi^*(\U_K)$ instead of $\U_G$, but this only works if $\varphi$ is injective because otherwise $\varphi^*(\U_G)$ is not complete.

This issue can be resolved by carrying an embedding $\varphi^*(\U_K)\hookrightarrow \U_G$ around as an additional datum with respect to which one forms the restriction, as we now explain.

\subsection{Restriction maps}
\label{sec:restrictions}
Let $\Fin_{\U}$ denote the category of pairs $(G,\U_G)$ of a finite group $G$ together with a complete $G$-set universe $\U_G$, in which a morphism $(\varphi,\alpha)$ from $(G,\U_G)$ to $(K,\U_K)$ is a group homomorphism $\varphi:G\to K$ and a $G$-equivariant embedding $\alpha:\varphi^*(\U_K)\hookrightarrow \U_G$.

Now we let $X$ be a symmetric spectrum and $(\varphi:G\to K,\alpha:\varphi^*(\U_K)\hookrightarrow \U_G)$ a morphism in $\Fin_{\U}$. Further let $x$ be an element of $\pi_0^{K,\U_K}(X)$ represented by a $K$-map $f:S^M\to X(M)$ with $M\subseteq \U_K$. Then we define $(\varphi,\alpha)^*(x)\in \pi_0^{G,\U_G}(X)$ as the class of the composite
\[ S^{\alpha(M)}\xr{S^{(\alpha_{|M})^{-1}}} S^M \xr{f} X(M) \xr{X(\alpha_{|M})} X(\alpha(M)).
\]
This class does not depend on the chosen representative $f$ and hence we obtain a restriction map
\[(\varphi,\alpha)^*:\pi_0^{K,\U_K}(X)\to \pi_0^{G,\U_G}(X).\]
The following is straightforward:
\begin{Prop} For every symmetric spectrum $X$ the assignment
\begin{eqnarray*} (G,\U_G) & \mapsto & \pi_0^{G,\U_G}(X) \\
									(G\xr{\varphi} K,\varphi^*(\U_K)\stackrel{\alpha}{\hookrightarrow} \U_G) & \mapsto & ((\varphi,\alpha)^*:\pi_0^{K,\U_K}(X)\to \pi_0^{G,\U_G}(X))
\end{eqnarray*}
defines a contravariant functor $\upi_0(X)$ from $\Fin_{\U}$ to abelian groups.
\end{Prop}
Using the suspension isomorphisms $\pi_n^{G,\U_G}(X)\cong \pi_0^{G,\U_G}(\Omega^n(X))$ for $n\geq 0$ as well as $\pi_n^{G,\U_G}(X)\cong \pi_0^{G,\U_G}(S^{-n}\wedge X)$ for $n<0$ we obtain natural $\Fin_\U^{op}$-functors $\upi_n(X)$ for all $n\in \mathbb{Z}$.

We note the following special cases of operations obtained this way:
\begin{enumerate}[(i)]
 \item Every subgroup inclusion $i_H^G:H\leq G$ gives rise to a restriction homomorphism  \[(i_H^G)^*:\pi_0^{G,\U_G}(X)\to \pi_0^{H,(i_H^G)^*(\U_G)}(X). \]
by applying the above construction to the morphism $(i_H^G,\id):(H,(i_H^G)^*(\U_G))\to (G,\U_G)$ in $\Fin_\U$.
	\item Every surjective group homomorphism $\varphi:G\twoheadrightarrow K$ gives rise to a restriction homomorphism
	\[ (\varphi,(- \circ \varphi))^*:\pi_0^{K,\N^K}(X)\to \pi_0^{G,\N^G}(X), \] 	
where $\N^G$ denotes the complete $G$-set universe of functions from $G$ to the natural numbers (and likewise for $K$) and $(-\circ \varphi)$ denotes the induced injective map by precomposing with $\varphi$.
 \item Every pair of a subgroup $i_H^G:H\leq G$ and an element $g\in G$ induces a conjugation homomorphism
    \[ c_g^*:\pi_0^{H,(i_H^G)^*(\U_G)}(X)\to \pi_0^{gHg^{-1},(i_{gHg^{-1}}^G)^*(\U_G)}(X) \]
by applying the above construction to the morphism \[ (g^{-1}(-)g,g\cdot -):(gHg^{-1},(i_{gHg^{-1}}^G)^*(\U_G))\to (H,(i_H^G)^*(\U_G)). \]
 \item Every injective $G$-equivariant self-map $\alpha:\U_G\hookrightarrow \U_G$ gives rise to an endomorphism \[ \alpha\cdot -:\pi_0^{G,\U_G}(X)\to \pi_0^{G,\U_G}(X)\]
via $(\id,\alpha)^*$. This defines an additive natural left action of the monoid $\Inj_G(\U_G,\U_G)$ on $\pi_0^{G,\U_G}(X)$.
\end{enumerate}
Any morphism in $\Fin_{\U}$ can be written as a composite of those of type (i),(ii) and (iv). The first three should be seen as genuine global equivariant operations which survive to the global homotopy category, whereas non-triviality of the $\Inj_G(\U_G,\U_G)$-action implies that the morphism $X\to QX$ is not a global $\upi_*$-isomorphism and hence the $\pi_n^{G,\U_G}(X)$ are not derived (cf. Proposition \ref{prop:semistable}). In the non-equivariant case $(\{e\},\N)$ this action was examined in \cite{Sch08}, the equivariant version $(G,\U_G)$ in \cite{Hau14}.

We also included the conjugation maps above because they allow a cleaner description of the double coset formula in Proposition \ref{prop:doublecoset}. They have the following property:
\begin{Lemma} \label{lem:conjtrivial} All inner conjugations $c_g^*$ act as the identity on $\pi_0^{G,\U_G}(X)$. 
\end{Lemma}
\begin{proof} Let $x\in \pi_0^{G,\U_G}(X)$ be an arbitrary element, represented by a $G$-map $f:S^M\to X(M)$ for some finite $M\subseteq \U_G$. Then, by definition, $c_g^*(x)$ is the class represented by the composite
\[ S^{\varphi^*(M)}\xr{g^{-1}\cdot -} S^M \xr{f} X(M) \xr{X(g\cdot -)} X(M). \]
The map $X(g\cdot -):X(M)\to X(M)$ is equal to multiplication by $g$. So, since $f$ is $G$-equivariant, this composite equals $f$ and hence $c_g^*(x)=c_g^*([f])=[f]=x$, which proves the claim. 
\end{proof}

\begin{Remark} The category $\Fin_{\U}$ comes with a forgetful functor to the category $\Fin$ of finite groups. The functor is surjective on objects and morphisms, but it does \emph{not} have a section. In fact, for any non-trivial finite group $G$, there do not exist two lifts of the homomorphisms $i:\{e\}\to G$ and $p:G\to \{e\}$ such that their composite is the identity. This is because the second component of any preimage $(p:G\to \{e\},p^*(\U_{\{e\}})\hookrightarrow \U_G)$ is never surjective, since the $G$-set universe $p^*(\U_{\{e\}})$ is trivial. Hence, the second component of the composite is also not surjective, in particular not the identity. There are symmetric spectra $X$ for which $(\id_{\{e\}},\alpha:\U_{\{e\}}\hookrightarrow \U_{\{e\}})$ does not act surjectively on $\pi_0^{\{e\},\U_{\{e\}}}(X)$ for every $\alpha$ which is not surjective (this is the case in Example \ref{sec:notsemistable}), hence this shows that there is in general no way to turn the $\Fin_{\U}^{op}$-functor $\upi_0(X)$ into 
a $\Fin^{op}$-
functor.
\end{Remark}

\subsection{Transfer maps} \label{sec:transfers} The assignment $(G,\U_G)\mapsto \pi_0^{G,\U_G}(X)$ has more structure than that of a $\Fin_{\U}^{op}$-functor, it also allows \emph{transfer maps} of the form $\tr_H^G:\pi_0^{H,(i_H^G)^*(\U_G)}(X)\to \pi_0^{G,\U_G}(X)$ for a subgroup $H$ of $G$ and the restricted (complete) $H$-set universe $(i_H^G)^*(\U_G)$. The construction and properties of these transfer maps are similar to those for orthogonal spectra, so we will be brief (cf. \cite[Constr. III.3.14]{Sch15}).

Transfer maps are based on the following construction: Let $M\subseteq \U_G$ be a $G$-subset which contains a copy of $G/H$. By thickening up the embedding $G/H\hookrightarrow M\hookrightarrow \R^M$ we obtain another $G$-embedding $G\ltimes_H D(\R^M)\hookrightarrow \R^M$, where $D(-)$ denotes the closed unit disc. Collapsing everything outside the image of the interior of $G\ltimes_H D(\R^M)$ to a point yields a map $p_H^G:S^M\to G\ltimes_H S^M$, the ``Thom-Pontryagin collapse map''.

Now let $X$ be a symmetric spectrum of spaces and $x\in \pi_0^{H,i^*(\U_G)}(X)$ an element represented by an $H$-map $f:S^M\to X(M)$. Without loss of generality we can assume that $M$ is in fact a $G$-subset of $\U_G$ and allows a $G$-embedding of $G/H$. Then the transfer $\tr_H^G(x)\in \pi_0^{G,\U_G}(X)$ is defined as the class of the composite
\[ S^M\xr{p_H^G}G\ltimes_H S^M\xr{G\ltimes_H f} G\ltimes_H X(M)\xr{\mu} X(M), \]
where $\mu$ is the action map (which uses that $X(M)$ is a $G$-space).
\begin{Prop} \label{prop:doublecoset} The transfer maps $\tr_H^G$ do not depend on the choice of embedding $G/H\hookrightarrow \U_G$. They are additive and functorial in subgroup inclusions. Furthermore, they are related to the restriction maps by the following formulas:\begin{enumerate}[(i)]
	\item For every morphism $(\varphi:G\twoheadrightarrow K,\alpha:\varphi^*(\U_K)\hookrightarrow \U_G)$ in $\Fin_{\U}$ with surjective $\varphi$ and every subgroup $i:L\leq K$, the relation 
\[  (\varphi,\alpha)^*\circ {\tr}_L^K = {\tr}_{\varphi^{-1}(L)}^G\circ (\varphi_{|\varphi^{-1}(L)}:\varphi^{-1}(L)\to L,\alpha)^*\]
holds as maps $\pi_0^{L,(i_L^K)^*(\U_K)}(X)\to \pi_0^{G,\U_G}(X)$.
	\item For every pair of subgroups $H,J\leq G$ the double coset formula
 \[ (i_J^G)^*\circ {\tr}_H^G = \sum_{[g]\in J\backslash G/H} {\tr}^J_{J\cap gHg^{-1}}\circ c_g^*\circ (i_{g^{-1}Jg \cap H}^H)^* 
 \]	
holds.
\end{enumerate}
\end{Prop}
\begin{proof} \cite[Prop. 2.34 and 2.35, Formula 3.15]{Sch15} for orthogonal spectra.
\end{proof}
Since every morphism $(\varphi,\alpha)$ in $\Fin_{\U}$ can be written as the composite of a morphism of type $(i)$ and a subgroup inclusion as in $(ii)$, these two can be combined to give a general formula describing the interaction between restrictions and transfers. Again, the definition of the transfer maps is extended to $\upi_n(X)$ via the suspension isomorphisms. 

\subsection{Semistability} \label{sec:semistability} In these terms, a $\mathcal{F}in$-\emph{global functor} in the sense of \cite{Sch15} (or, equivalently, an \emph{inflation functor} in the sense of \cite{We93}) can be described as a $\Fin_{\U}^{op}$-functor with transfers satisfying the relations of Lemma \ref{lem:conjtrivial} and Proposition \ref{prop:doublecoset} and for which the $\Fin_{\U}^{op}$-part factors through $\Fin^{op}$, i.e., for which the action of an element $(\varphi,\alpha)$ does not depend on the $\alpha$ (cf. \cite[Rmks. IV.4.23 and IV.4.24]{Sch15}). This leads to the following definition:
\begin{Def}[Global semistability]
A symmetric spectrum $X$ is called \emph{globally semistable} if the $\Fin_{\U}^{op}$-functor $\upi_n(X)$ factors through a $\Fin^{op}$-functor for every $n\in \mathbb{Z}$.
\end{Def}
Then the previous discussion implies:
\begin{Prop} If $X$ is globally semistable, the homotopy groups $\pi_*^{G,\U_G}(X)$ only depend on $\U_G$ up to canonical isomorphism (hence they can be denoted by $\pi_*^G(X)$) and the collection $\upi_*(X)=\{\pi_*^G(X)\}_{\text{G finite}}$ naturally forms a $\mathcal{F}in$-global functor. 
\end{Prop}
The class of globally semistable symmetric spectra includes a lot of examples and is closed under many operations, as the following proposition shows. For $(i)$ we recall from \cite[Def. 3.29]{Hau14} (and the remark preceding it) that a $G$-symmetric spectrum $X$ is called $G$-semistable if the $\Inj_H(\U_H,\U_H)$-action on $\pi_n^{H,\U_H}(X)$ is trivial for all $n\in \mathbb{Z}$ and all subgroups $H\leq G$.
\begin{Prop} \label{prop:semistable} The following hold:
\begin{enumerate}[(i)]
 \item A symmetric spectrum is globally semistable if and only if it is $G$-semistable for every finite group $G$.
 \item Global $\Omega$-spectra are globally semistable.
 \item Every symmetric spectrum underlying an orthogonal spectrum is globally semistable.
 \item Every symmetric spectrum $X$ for which every homotopy group $\pi_n^{G,\U_G}(X)$ is a finitely generated abelian group is globally semistable.
 \item The smash product of two globally semistable symmetric spectra is again globally semistable, as long as at least one smash factor is flat.
 \item A symmetric spectrum is globally semistable if and only if the morphism $q_X:X\to QX$ is a global $\upi_*$-isomorphism; in other words if and only if the map from the naive to the derived equivariant homotopy groups is an isomorphism.
 \item A morphism between globally semistable symmetric spectra is a global equivalence if and only if it is a global $\upi_*$-isomorphism.
\end{enumerate}
\end{Prop}
\begin{proof} $(i)$: The ``only if'' part is clear. The other direction follows from the fact that given a group homomorphism $\varphi:G\to K$ and two $G$-embeddings $\alpha_1,\alpha_2:\varphi^*(\U_K)\hookrightarrow \U_G$, there exist $\beta_1,\beta_2\in \Inj_G(\U_G,\U_G)$ such that $\beta_1\circ \alpha_1=\beta_2\circ \alpha_2$.

Using $(i)$, items $(iii), (iv), (v)$ and $(vii)$ follow from \cite[Thm. 3.47]{Hau14}. Moreover, every global $\Omega$-spectrum can be replaced by a $G\Omega$-spectrum up to eventual level equivalence (as explained after Definition \ref{def:globalomega}), in particular up to $\upi^{\U_G}_*$-isomorphism. Hence, \cite[Thm. 3.47]{Hau14} also implies $(ii)$. If $q_X:X\to QX$ is a global $\upi_*$-isomorphism, then $X$ is globally semistable, since we just argued that $QX$ is globally semistable. If in turn $X$ is assumed to be globally semistable, we know that the global equivalence $q_X:X\to QX$ must be a global $\upi_*$-isomorphism by $(vii)$. This gives $(vi)$ and finishes the proof.
\end{proof}
\subsection{Example} \label{sec:notsemistable} We close this section with an example of a symmetric spectrum which is not globally semistable, the free symmetric spectrum $F_1^{\{e\}} S^1$. There is a natural $G$-isomorphism $(F_1^{\{e\}} S^1)(M)\cong M_+\wedge S^M$ (cf. \cite[Ex. 3.35]{Hau14}) which implies that
\begin{align*} \pi_0^{G,\U_G}(F_1^{\{e\}} S^1) & \cong \colim_{M\subseteq \U_G} [S^M,M_+\wedge S^M]^G \\ & \cong \colim_{M\subseteq \U_G}[S^M,(\U_G)_+\wedge S^M]^G \\ & \cong \pi_0^{G,\U_G}(\Sigma^{\infty}_+(\U_G)), \end{align*}
with $G$ acting on $\U_G$. The tom Dieck-splitting shows that this is a free abelian group with basis $\{\tr_H^G(x)\}$, where $(H,x)$ runs through representatives of $G$-conjugacy classes of pairs of a subgroup $H$ of $G$ and an $H$-fixed point $x$ of $(i_H^G)^*(\U_G)$.

Focusing on those basis elements which are not a transfer from a proper subgroup, we see:
\begin{Cor} The $\Fin_{\U}^{op}$-functor $\upi_0(F_1^{\{e\}} (S^1))$ contains the subfunctor
\begin{eqnarray*} (G,\U_G)& \mapsto & \mathbb{Z} [(\U_G)^G] \\
		  (\varphi:G\to K,\alpha:\varphi^*(\U_K)\hookrightarrow \U_G)& \mapsto & (\mathbb{Z}[(\U_K)^K]\hookrightarrow \mathbb{Z}[(\varphi^*(\U_K))^G]\xr{\mathbb{Z}[\alpha]} \mathbb{Z}[(\U_G)^G]).
\end{eqnarray*}
\end{Cor}
This determines the whole $\Fin_{\U}^{op}$-functor structure on $\upi_0(F_1^{\{e\}} S^1)$ via Proposition \ref{prop:doublecoset}. The action of a morphism $(\varphi,\alpha)$ in $\Fin_{\U}$ very much depends on the $\alpha$ and hence $F_1^{\{e\}} (S^1)$ is not globally semistable.

\section{Comparison to orthogonal spectra}
\label{sec:comparison}
In this section we show that global homotopy theory of symmetric spectra is equivalent to $\mathcal{F}in$-global homotopy theory of orthogonal spectra in the sense of \cite{Sch15}. For this we quickly recall the relevant definitions in the orthogonal context.

\begin{Def}[Orthogonal spectra] An orthogonal spectrum is a collection of based $O(n)$-spaces $\{X_n\}_{n\in \N}$ with structure maps $X_n\wedge S^1\to X_{n+1}$ whose iterates $X_n\wedge S^m\to X_{n+m}$ are $(O(n)\times O(m))$-equivariant.
\end{Def}
An orthogonal spectrum $X$ can be evaluated on $G$-representations $V$ via the formula $X_n\wedge_{O(\dim(V))} \Lin(\R^{\dim(V)},V)_+$, with $G$-acting through $V$ (where $\Lin(\R^{\dim(V)},V)$ denotes the space of linear isometries). Again, these are connected by $G$-equivariant generalized structure maps of the form $X(V)\wedge S^W\to X(V\oplus W)$.

Every orthogonal spectrum $X$ has an underlying symmetric spectrum of spaces $U(X)$ by restricting the $O(n)$-action on $X_n$ to a  $\Sigma_n$-action along the embedding as permutation matrices. The resulting restriction functor $U:Sp^O\to Sp^{\Sigma}$ has a left adjoint $L$, formally obtained via a left Kan extension (cf. \cite[Secs. I.3 and III.23]{MMSS01} for details).

\begin{Example} For a finite $G$-set $M$ there is a natural $G$-homeomorphism $U(X)(M)\cong X(\R^M)$ induced by linearizing a bijection $m\xr{\cong} M$ to a linear isometry $\R^m\xr{\cong} \R^M$.
\end{Example}
Using this $G$-homeomorphism, it is not hard to see that the equivariant homotopy groups of an orthogonal spectrum as defined in \cite[Sec. III.2]{Sch15} are isomorphic to those of the underlying symmetric spectrum defined in Section \ref{sec:homgroups}. Combining this with Proposition \ref{prop:semistable} we see that for a morphism of orthogonal spectra $f:X\to Y$ the following are equivalent:
\begin{itemize}
 \item $f$ is a $\mathcal{F}in$-equivalence in the sense of \cite[Def. IV.1.15]{Sch15}.
 \item $U(f)$ is a global $\upi_*$-isomorphism of symmetric spectra.
 \item $U(f)$ is a global equivalence of symmetric spectra.
\end{itemize}

Around this notion of equivalence Schwede defines the $\mathcal{F}in$-global model structure on orthogonal spectra (\cite[Thm. IV.1.19]{Sch15}). We have:
\begin{Theorem} \label{theo:quillen} The adjunction \[ L:Sp^{\Sigma}_{\T}\rightleftarrows Sp^O:U \] is a Quillen equivalence for the global model structure on symmetric spectra and the $\mathcal{F}in$-global model structure on orthogonal spectra.
\end{Theorem}
\begin{proof} It is straightforward to see (using the natural isomorphism of evaluations described above) that the adjunction becomes a Quillen pair for the respective level model structures (the orthogonal one is defined in \cite[Prop. IV.1.15]{Sch15}). Applying \cite[Thm. 5.2] {Hau14} for every finite group $G$ we see that $L$ furthermore sends flat cofibrations which are also global equivalences to $\F in$-equivalences, hence $L$ becomes a left Quillen functor for the stable model structures and thus $(L,U)$ a Quillen pair.

Hence, it remains to show that the adjunction induces an equivalence between the homotopy categories. Since $U$ preserves and reflects weak equivalences, it suffices to show that for every flat symmetric spectrum $X$ the morphism $X\to U(L(X))$ is a global equivalence. But since every flat symmetric spectrum $X$ is $G$-flat when given the trivial $G$-action, this follows from \cite[Thm. 5.2]{Hau14}.
\end{proof}

\section{Examples} \label{sec:exa}
Every orthogonal spectrum can be restricted to a symmetric spectrum, so all examples in \cite{Sch15} also give examples for symmetric spectra and their global behavior. In this section we list some constructions of symmetric spectra (from the point of view of global homotopy theory) that do not arise from orthogonal spectra.

\subsection{Suspension spectra of $\I$-spaces} There is an unstable analog of symmetric spectra, called $\I$-spaces. Again, these were previously considered as a model for unstable non-equivariant homotopy theory (see, for example, \cite{SS12}, \cite{SS13} and \cite{Lind13}). They come with a Day convolution product, the commutative monoids over which model $E_{\infty}$-spaces.

In \cite[Sec. I.7]{Sch15} Schwede describes a global equivariant point of view on $\I$-spaces, which we quickly recall. Let $\I$ denote the category of finite sets and injective maps.
\begin{Def} An \emph{$\I$-space} is a functor from $\I$ to the category of spaces. \end{Def}
Let $A$ be an $\I$-space. By functoriality, if a finite set $M$ comes equipped with an action of a finite group $G$, the evaluation $A(M)$ becomes a $G$-space. Every injection of $G$-sets $M\hookrightarrow N$ induces a $G$-equivariant map $A(M)\to A(N)$. In \cite[Prop. I.7.17]{Sch15}, Schwede shows that there is a level model structure on $\I$-spaces where the weak equivalences and fibrations are those morphisms that become $G$-weak equivalences respectively $G$-fibrations on $-(M)$ for all finite groups $G$ and finite $G$-sets $M$.

An $\I$-space $A$ is called \emph{static} if for every injection $M\hookrightarrow N$ of faithful finite $G$-sets the induced map $A(M)^G\to A(N)^G$ is a weak equivalence. A morphism of $\I$-spaces is a \emph{global equivalence} if it induces bijections on all hom-sets into static $\I$-spaces in the level homotopy category. Together with the level cofibrations, these form the global model structure for $\I$-spaces of \cite[Thm. I.7.19]{Sch15}.

For a static $\I$-space $A$, the evaluation $A(M)$ at a faithful finite $G$-set $M$ should be thought of as the \emph{$G$-space underlying $A$}. By the definition of static its $G$-homotopy type does not depend on the choice of $M$. The $G$-space underlying an arbitrary $\I$-space $A$ is not as easy to describe directly, but it can be defined by first replacing by a globally equivalent static $\I$-space $QA$ and then taking the underlying $G$-space of $QA$. In this sense a global equivalence can be interpreted as a morphism that induces equivalences on all underlying $G$-spaces.

Every $\I$-space $A$ gives rise to a suspension symmetric spectrum of spaces $\Sigma^{\infty}_+ A$. Its $n$-th level is given by $A(\underline{n})_+\wedge S^n$ with diagonal $\Sigma_n$-action, the structure map \[ (A(\underline{n})_+\wedge S^n)\wedge S^1\to A(\underline{n+1})_+\wedge S^{n+1}\]
is the smash product of the induced map $A(\underline{n}\hookrightarrow \underline{n+1})$ with the associativity isomorphism $S^n\wedge S^1\cong S^{n+1}$. This construction is left adjoint to $\Omega^{\infty}:Sp^{\Sigma}_{\T}\to \I\text{-spaces}$ defined by $(\Omega^{\infty}(X))(\underline{n})\defeq \Omega^n X_n$. Since $\Omega^{\infty}$ turns global $\Omega$-spectra into static $\I$-spaces, it is not hard to see that the adjunction $(\Sigma^{\infty}_+,\Omega^{\infty})$ becomes a Quillen pair for the respective global model structures.

Let $A$ be a cofibrant static $\I$-space. One can show that the $G$-homotopy type of $\Sigma^{\infty}_+ A$ is that of the suspension spectrum of the underlying $G$-space of $A$ in the sense described above. Hence, suspension spectra of $\I$-spaces assemble various equivariant suspension spectra into one global object.
\begin{Remark} In all of the above one can alternatively consider functors from $\I$ to the category of simplicial sets. Then the analogous statements hold.
\end{Remark}

\begin{Example}[Global classifying spaces] Let $G$ be a finite group and $M$ a finite $G$-set. This data gives rise to an $\I$-space $\I(M,-)/G$ whose evaluation on a finite set $N$ is the set of injective maps from $M$ to $N$, modulo the $G$-action by pre-composition. Giving a morphism from $\I(M,-)/G$ to an $\I$-space $A$ is equivalent to picking a $G$-fixed point in the evaluation $A(M)$. So - by definition of the notion of global equivalence - the global homotopy type of $\I(M,-)/G$ is the same for all \emph{faithful} $G$-sets $M$. The $\I$-spaces $\I(M,-)/G$ for faithful $M$ are called \emph{global classifying spaces of $G$}. Given another finite group $K$, the $K$-space underlying $\I(M,-)/G$ is a classifying space for principal $G$-bundles in $K$-spaces, cf. \cite[Rmk. I.2.13]{Sch15}. This is easiest to see via the equivalence in \cite[Thm. I.7.26]{Sch15}.

Ranging through all finite groups $G$, the suspension spectra of global classifying spaces of finite groups (which are isomorphic to global free spectra of the form $F^G_{M}S^M$) form a set of compact generators of the triangulated $\mathcal{F}in$-global stable homotopy category.
\end{Example}

\subsection{Ultracommutative localizations} Let $A\subseteq \mathbb{Q}$ be a subring, $M(A,1)$ a Moore space for $A$ in degree $1$ and $i:S^1\to M(A,1)$ a map inducing the inclusion $\mathbb{Z}\hookrightarrow A$ on first homology. We define a symmetric spectrum $MA$ via $MA_n=M(A,1)^{\wedge n}$ with permutation $\Sigma_n$-action and structure map \[ M(A,1)^{\wedge n}\wedge S^1\xr{\id\wedge i} M(A,1)^{\wedge (n+1)}. \]
The associativity homeomorphisms $M(A,1)^{\wedge n}\wedge M(A,1)^{\wedge m}\cong M(A,1)^{\wedge (n+m)}$ together with the equality $S^0= M(A,1)^{\wedge 0}$ give $MA$ the structure of an ultracommutative symmetric ring spectrum.

To determine the global homotopy type of $MA$ we note that the map $M(A,1)\wedge S^1\xr{\id\wedge i} M(A,1)^{\wedge 2}$ is a weak equivalence of spaces, since $A\otimes \mathbb{Z}\to A\otimes A$ is an isomorphism. So, given a subgroup $H\leq \Sigma_n$, the map
\[ M(A,1)\wedge (S^n)^H\cong M(A,1)\wedge S^{\wedge (\underline{n}/H)}\xr{(\id\wedge i^{\wedge (\underline{n}/H)})} M(A,1)\wedge M(A,1)^{\wedge (\underline{n}/H)}\cong M(A,1)\wedge (M(A,1)^{\wedge n})^H \]
is also a weak equivalence. In other words, the morphism $\Sigma^{\infty} (M(A,1))\to \sh MA$ adjoint to the identity of $M(A,1)$ is a global level equivalence. The same argument also shows that $\alpha_{MA}:S^1\wedge MA\to \sh MA$ is a positive global level equivalence. So we find that $MA$ is globally equivalent to a desuspension of the suspension spectrum of $M(A,1)$ and hence its global homotopy type is that of the homotopy colimit of the sequence
\[ \mathbb{S}\xr{\cdot n_1} \mathbb{S}\xr{\cdot n_2} \mathbb{S}\xr{\cdot n_3}  \hdots \]
where the $n_i$ range through the elements of $\mathbb{Z}$ that become inverted in $A$. Thus, the (derived) smash product $-\wedge MA$ computes the $A$-localization in the global homotopy category. On equivariant homotopy groups it has the effect of tensoring with $A$.
 
In particular, the ultracommutative structure on $MA$ can be used to see that arithmetic localizations of ultracommutative symmetric ring spectra are again ultracommutative symmetric ring spectra, which is not a priori clear and does not hold in general for equivariant localizations (cf. \cite{HiHo14}, in particular Section 4.1).

\begin{Remark}
The construction of $MA$ above works more generally for any based space $X$ together with a based map $S^1\to X$. This gives a functor from the category of based spaces under $S^1$ to ultracommutative ring spectra, which is left adjoint to sending an ultracommutative ring spectrum $Z$ to the unit map $S^1\to Z_1$. The latter is a right Quillen functor for the positive global model structure and the usual Quillen model structure on spaces under $S^1$, turning the adjunction into a Quillen pair. In fact, the adjunction is already a Quillen pair if one uses the non-equivariant positive projective model structure on commutative symmetric ring spectra (as constructed in \cite[Thm. 15.1]{MMSS01}).
This implies that the ultracommutative ring spectra that arise through this construction are multiplicatively left-induced from non-equivariant commutative ring spectra in the sense of Appendix~\ref{sec:appendix}.
\end{Remark}

\subsection{Global algebraic $K$-theory} In \cite{Sch13alg} Schwede introduces a symmetric spectrum model for global (projective or free) algebraic $K$-theory of a ring $R$. Below we summarize the free version. In fact we give a slight variation of that of \cite{Sch13alg}, as we explain in Remark \ref{rem:difference}.

Let $R$ be a discrete ring. Each level $kR(M)$ is the geometric realization of a simplicial space $\{kR(M)_n\}_{n \in \N}$ that we now explain. A $0$-simplex of this simplicial space is represented by a finite unordered labeled configuration $(W_1,\hdots,W_k;x_1,\hdots,x_k)$ of the following kind:
\begin{itemize}
	\item The $x_i$ are points in the sphere $S^M$.
	\item The $W_i$ are finitely generated free submodules of the polynomial ring $R[M]$ with variable set $M$, such that their sum is direct and the inclusion $W_1\oplus \hdots \oplus W_k \hookrightarrow R[M]$ allows an $R$-linear splitting.
\end{itemize}
These configurations are considered up to the equivalence relation that a labeled point $(W_i,x_i)$ can be left out if either $W_i$ is zero or $x_i$ the basepoint, and that if two $x_i$ are equal they can be replaced by a single one with label the sum of the previous labels. The $\Sigma_M$-action is the diagonal one through its actions on $S^M$ and $R[M]$.

General $n$-simplices are given by similar equivalence classes of configurations, where instead of a single free submodule $W_i$, each point $x_i$ carries an $n$-chain of $R$-module isomorphisms $(W_{i_0}\xr{\cong} W_{i_1}\xr{\cong} \hdots \xr{\cong} W_{i_n})$, such that for every $0\leq j\leq n$ the tuple $(W_{1_j},\hdots,W_{k_j})$ satisfies the conditions above. The simplicial structure maps are the usual ones from the nerve and do not affect the $x_i$. The spectrum structure maps $kR(M)\wedge S^N\to kR(M\sqcup N)$ are given by smashing the configurations with an element of $S^N$ and leaving the labels unchanged.

In \cite{Sch13alg} Schwede shows the following:

\begin{itemize}
	\item The symmetric spectrum $kR$ is globally semistable.
	\item Its fixed point spectrum (cf. \cite[Sec. 6]{Sch13alg}) represents the direct sum $K$-theory of $R[G]$-lattices, i.e., $R[G]$-modules that are finitely-generated free as $R$-modules. In particular, the equivariant homotopy groups $\pi_*^G(kR)$ are the $K$-groups of $R[G]$-lattices.
	\item If $R$ is commutative, the smash product of modules gives $kR$ the structure of an ultracommutative symmetric ring spectrum.
\end{itemize}

If $R$ satisfies dimension invariance, the spectrum $kR$ comes with a natural filtration: Let $kR^n(M)$ be the subspace of $kR(M)$ of those configurations $(W_1,\hdots,W_k;x_1,\hdots,x_k)$ where the sum of the $R$-ranks of the $W_i$ is at most $n$, and similarly for higher simplices. These subspaces are closed under the simplicial and spectrum structure and thus define a symmetric subspectrum $kR^n$. This gives a filtration
\[ *=kR^0\to kR^1\to \hdots \to kR=\colim_{n\in \N} kR^n. \]
The underlying non-equivariant filtration is studied by Arone and Lesh in \cite{AL10}, where they call it the \emph{modified stable rank filtration} of algebraic $K$-theory. In joint work with Dominik Ostermayr \cite{HO15} we extend some of their results to the global context to show that the subquotients $kR^n/kR^{n-1}$ are globally equivalent to suspension spectra of certain $\I$-spaces associated to the lattice of non-trivial direct sum decompositions of $R^n$. This can be used to give an algebraic description of the $\mathcal{F}in$-global functors $\pi_0^G(kR^n)$.

\begin{Remark} \label{rem:difference} The version of $kR$ we described here differs slightly from the one in \cite{Sch13alg}. There the tuple $(W_1,\hdots,W_k)$ has to satisfy the additional property that for every monomial $t=\prod_{m\in M} m^{i_m}\in R[M]$ there is at most one $i$ such that $W_i$ contains an element whose $t$-component is non-trivial (which in that setup in particular guarantees that the sum of the $W_i$ is direct). The inclusion from the $kR$ in \cite{Sch13alg} to the one above is a global level equivalence.
\end{Remark}

\appendix

\section{Model structures with respect to families} \label{sec:appendix}
In this appendix we explain how to construct model structures with respect to global families of finite groups. For every such family we define two model structures, a projective and a flat one, both useful for constructing derived adjunctions. In the case of the family of trivial groups (where the homotopy category is the non-equivariant stable homotopy category) the projective model structure equals the one in \cite[Sec. 5.1]{HSS00}, the flat model structure is the one introduced in \cite{Shi04}. For the global family of all finite groups the two model structures coincide.

\begin{Def}[Global family] A \emph{global family} is a non-empty class of finite groups which is closed under subgroups, quotients and isomorphism.
\end{Def}

Let $\F$ be a global family.
\begin{Def} A morphism $f:X\to Y$ of symmetric spectra is called \begin{itemize}
\item an \emph{$\F$-level equivalence} if $f_n^H:X_n^H\to Y_n^H$ is a weak equivalence for all subgroups $H\leq \Sigma_n$ which lie in $\F$.
\item a \emph{projective $\F$-level fibration} if $f_n^H:X_n^H\to Y_n^H$ is a Serre/Kan fibration for all subgroups $H\leq \Sigma_n$ which lie in $\F$. 
\item a \emph{projective $\F$-cofibration} if each latching map $\nu_n[f]:X_n\cup_{L_n (X)} L_n(Y)\to Y_n$ is a $\Sigma_n$-cofibration with relative isotropy in $\F$.
\item a \emph{flat $\F$-level fibration} if it has the right lifting property with respect to all flat cofibrations (as defined in Definition \ref{def:levmod}) that are also $\F$-level equivalences.
\end{itemize}
\end{Def}

Then the following two propositions can again be obtained via \cite[Prop. 2.30]{Hau14}.
\begin{Prop} The classes of $\F$-level equivalences, projective $\F$-level fibrations and projective $\F$-cofibrations define a cofibrantly generated, proper and monoidal model structure on the category of symmetric spectra.
\end{Prop}
\begin{Prop} The classes of $\F$-level equivalences, flat $\F$-level fibrations and flat cofibrations define a cofibrantly generated, proper and monoidal model structure on the category of symmetric spectra.
\end{Prop}
From the point of view of $\F$-global homotopy theory we have to remember the $G$-homotopy type of a symmetric spectrum for all groups $G$ in $\F$, which leads to the following definition of stable equivalence:
\begin{Def}[$\F$-global equivalences] A morphism $f:X\to Y$ is called an \emph{$\F$-global equivalence} if it is a $G$-stable equivalence (in the sense of Definition \ref{def:gstable}) for all groups $G\in \F$.
\end{Def}

A morphism of symmetric spectra is called a \emph{projective (flat) $\F$-fibration} if it has the left lifting property with respect to all morphisms that are projective $\F$-cofibrations (respectively flat cofibrations) and $\F$-equivalences. Then we have:
\begin{Prop} The classes of $\F$-global equivalences, projective $\F$-fibrations and projective $\F$-cofibrations determine a cofibrantly generated, proper and monoidal model structure on the category of symmetric spectra, called the \emph{projective $\F$-global stable model structure}.
\end{Prop}
\begin{Prop} The classes of $\F$-global equivalences, flat $\F$-fibrations and flat cofibrations determine a cofibrantly generated, proper and monoidal model structure on the category of symmetric spectra, called the \emph{flat $\F$-global stable model structure}.
\end{Prop}
Each of these model structures can be obtained via a left Bousfield localization of the respective level model structure. For example this can be done by applying the small object argument to the subset of those maps $i\square \overline{\lambda}_M^N$ used in Section \ref{sec:stablemod} that are associated to a finite group $G\in \F$ and finite $G$-sets $M$ and $N$ (of which $M$ is faithful). It follows that a symmetric spectrum is fibrant in either of the $\F$-global model structures if and only if it is fibrant in the respective level model structure and in addition an $\F$-global $\Omega$-spectrum, i.e., if it satisfies the condition in Definition \ref{def:globalomega} for all $G\in \F$ (instead of for all finite $G$).
The flat $\F$-global model structure can also be obtained by left Bousfield localizing the full global model structure.

Since every projective $\F$-cofibration is a flat cofibration, the $\F$-global model structure and the flat $\F$-global model structure are Quillen equivalent via the identity adjunction. Furthermore, the same proof as that of Theorem \ref{theo:quillen} applies to show that the projective $\F$-model structure is Quillen equivalent to the $\F$-global model structure on orthogonal spectra as introduced in \cite[Thm. IV.1.19]{Sch15}.

Let $\F'\subseteq \F$ be an inclusion of global families of finite groups. Then, by definition, every $\F$-global equivalence is an $\F'$-global equivalence and hence the localization $Sp^{\Sigma}\to Sp^{\Sigma}[\F\text{-global eq.}^{-1}]$ factors uniquely through a functor $Sp^{\Sigma}[\F\text{-global eq.}^{-1}]\to Sp^{\Sigma}[\F'\text{-global eq.}^{-1}]$. This functor has both a left and a right adjoint (both fully faithful) obtained by deriving the identity adjunction with respect to the projective respectively flat model structures. In particular, this defines two functors from the non-equivariant stable homotopy category to the global stable homotopy category.
It can be shown (\cite[Exs. IV.5.11 and IV.5.28]{Sch15}) that the right adjoint gives rise to Borel theories, whereas the image of the left adjoint is given by symmetric spectra with constant geometric fixed points.

Finally, both the projective $\F$-global stable model structure and the flat $\F$-global stable model structure lift to categories of modules over a symmetric ring spectrum and algebras over a commutative symmetric ring spectrum. There exist positive versions of both model structures which lift to the category of commutative algebras over a commutative symmetric ring spectrum. These allow the construction of ``multiplicative'' change of family functors, but there is a caveat: A positive projective $\mathcal{F}$-cofibrant commutative symmetric ring spectrum is in general not projective $\mathcal{F}$-cofibrant as a symmetric spectrum if $\mathcal{F}$ is not the family of all finite groups. As a consequence, the underlying symmetric spectrum of a left-induced ultracommutative symmetric ring spectrum is in general not left-induced.

%\bibliographystyle{alpha}

%\bibliography{literature}

\address
\email

\end{document}